\documentclass[a4paper,11pt]{article}
\usepackage{titlesec}
\usepackage{amsthm}
\usepackage{amssymb}
\usepackage{color}
\usepackage{mathtools}
\usepackage{geometry}

\usepackage{algorithm2e}

\usepackage{verbatim}
\numberwithin{equation}{section}

\newtheorem{theorem}{Theorem}[section]

\newtheorem{lemma}[theorem]{Lemma}

\theoremstyle{definition}

\newtheorem{remark}[theorem]{Remark}
\newcommand{\bfs}[1]{{\boldsymbol #1}}

\usepackage{array}
\newcolumntype{C}[1]{>{\centering\arraybackslash}m{#1}}
\usepackage{graphics}
\usepackage{amssymb, latexsym, amsmath, amsfonts, epsfig}
\usepackage{bm}
\usepackage{graphicx}
\usepackage{color}
\usepackage{epstopdf}
\usepackage{comment}
\usepackage{soul}
\usepackage[normalem]{ulem}
\usepackage{float}
\usepackage{MnSymbol}

\newcommand{\polN}{\mathbb{N}}
\newcommand{\polP}{\mathbb{P}}
\newcommand{\polQ}{\mathbb{Q}}

\usepackage{caption}
\date{\vspace{-6ex}}

\begin{document}

\newcommand{\Question}[1]{{\marginpar{\color{blue}\footnotesize #1}}}
\newcommand{\blue}[1]{{\color{blue}#1}}
\newcommand{\red}[1]{{\color{red} #1}}

\newif \ifNUM \NUMtrue

\title{SoftFEM: revisiting the spectral finite element approximation of {second-order} elliptic operators}
\author{Quanling Deng\thanks{Department of Mathematics, University of Wisconsin--Madison, Madison, WI 53706, USA. E-mail addresses: quanling.deng@math.wisc.edu; qdeng12@gmail.com}
\and
Alexandre Ern\thanks{University Paris-Est, CERMICS (ENPC), 77455 Marne la Vall\'ee cedex 2, and INRIA Paris, 75589 Paris, France. E-mail address: alexandre.ern@enpc.fr}
}

\maketitle

\begin{abstract}
We propose, analyze mathematically, and study numerically a novel approach for the 
finite element approximation of the spectrum of second-order elliptic operators. 
{The main idea is to reduce the stiffness of the problem by subtracting a least-squares penalty on the gradient jumps across the mesh interfaces
from the standard stiffness bilinear form.}
This penalty bilinear form is similar to the known technique
used to stabilize finite element approximations in various contexts{. 
The penalty term is designed to dampen the high frequencies in the spectrum and so it is weighted here by a negative coefficient.
T}he resulting approximation technique is called softFEM {since it reduces the stiffness of the problem}. The two key advantages of
softFEM over the standard Galerkin FEM 
are to improve the approximation of the eigenvalues in the upper part of
the discrete spectrum and to reduce the condition number of the stiffness matrix.
We derive a sharp upper bound
on the softness parameter weighting the stabilization bilinear form so as to maintain
coercivity for the softFEM bilinear form. Then we prove that softFEM delivers the
same optimal convergence rates as the standard Galerkin FEM approximation for the 
eigenvalues and the eigenvectors. We next compare the discrete eigenvalues obtained 
when using Galerkin FEM and softFEM. Finally, a detailed analysis of linear softFEM
for the 1D Laplace eigenvalue problem delivers a sensible choice for the softness
parameter. With this choice, the stiffness reduction ratio scales linearly with the
polynomial degree. Various numerical experiments illustrate the benefits of using softFEM
over Galerkin FEM.
 \textbf{Mathematics Subjects Classification}: 65N15, 65N30, 65N35, 35J05
\end{abstract}
%


\section{Introduction} \label{sec:intr}

The optimal approximation of eigenvalues and eigenfunctions from second-order elliptic spectral problems by means of Galerkin finite element methods (FEM) is well-established. We refer the reader to the seminal contributions in Vainikko \cite{vainikko1964asymptotic,vainikko1967speed}, Bramble and Osborn \cite{bramble1973rate}, Strang and Fix \cite{strang1973analysis}, Osborn \cite{osborn1975spectral}, Descloux et al.~\cite{descloux1978spectral,descloux1978spectral2}, Babu\v{s}ka and Osborn \cite{babuvska1991eigenvalue}, and to the more recent reviews in \cite{boffi2010finite,Ern_Guermond_FEs_II_2021}. The approximation of elliptic spectral problems has also been studied by means of mixed finite element methods \cite{canuto1978eigenvalue, mercier1978eigenvalue, mercier1981eigenvalue}, discontinuous Galerkin methods \cite{antonietti2006discontinuous,giani2015hp}, hybridizable discontinuous Galerkin methods \cite{cockburn2010hybridization, gopalakrishnan2015spectral}, hybrid high-order methods \cite{calo2019spectral,carstensen2020guaranteed}, and virtual element methods \cite{gardini2017virtual}. All of these methods deliver optimally convergent approximations{. Since} the eigenfunctions become more and more oscillatory in the upper part of the spectrum, {their approximation is} accurate only in the lower part of the spectrum. In contrast, isogeometric analysis \cite{cottrell2006isogeometric} delivers a more accurate approximation in the upper part of the spectrum (see also \cite{deng2018dispersion,calo2019dispersion,deng2021boundary} for some recent improvements on the subject). 

The goal of this work is to improve on the Galerkin FEM spectral approximation so as to increase the accuracy in the upper part of the spectrum. This goal is achieved by reducing the stiffness of the discrete spectral problem. With this in mind, we refer to the newly coined method as \textit{softFEM}. {The idea is to subtract a least-squares penalty on the gradient jumps across the mesh interfaces from the standard stiffness bilinear form.} Thus, 
the softFEM bilinear form is defined as
\begin{equation}
\hat a(\cdot, \cdot) := a(\cdot, \cdot)  - \eta s(\cdot, \cdot),
\end{equation}
where $a(\cdot, \cdot)$ is the standard Galerkin FEM stiffness bilinear form, $\eta$ is the so-called \textit{softness} parameter, and $s(\cdot, \cdot)$ is the bilinear form penalizing the gradient jumps across the mesh interfaces. The idea behind softFEM shares some common ground with isogeometric analysis where the basis functions have at least $C^1$-smoothness. In softFEM, the same basis functions are used as in Galerkin FEM so that the smoothness is only $C^0$. However, by considering the bilinear form $\hat a(\cdot, \cdot)$ instead of $a(\cdot, \cdot)$, 
{one reduces the amount of energy stored in the gradient jumps of eigenfunctions associated with the large eigenvalues in the spectrum.
This change is not needed for eigenfunctions associated with the lower part of the spectrum since those eigenfunctions are smooth and can be accurately approximated on a given mesh.
}
We notice that the bilinear form $s(\cdot, \cdot)$ has been considered for the purpose of stabilization (i.e., leading to a positive contribution and not to a negative one as in the present work) in various contexts related, in particular, to advection-dominated advection-diffusion equations and to the Stokes equations \cite{Burman:05,Burman_Hansbo_2006,burman2007continuous}. 
{In the context of the Helmholtz equation, the bilinear form
$s(\cdot,\cdot)$ is weighted by a coefficient with positive imaginary part
to ensure coercivity \cite{Wu:14}. In addition, the possibility of using a weighting coefficient with negative real part has been considered in \cite{burman2016linear,DuWu:15} to improve the phase error.}
 Incidentally, we mention that the term softFEM has been used recently in \cite{PLTJ:19} in a completely different context related to heuristic optimization and soft computing for solid mechanics.

\begin{figure}[h!]
\centering
\includegraphics[height=5.2cm]{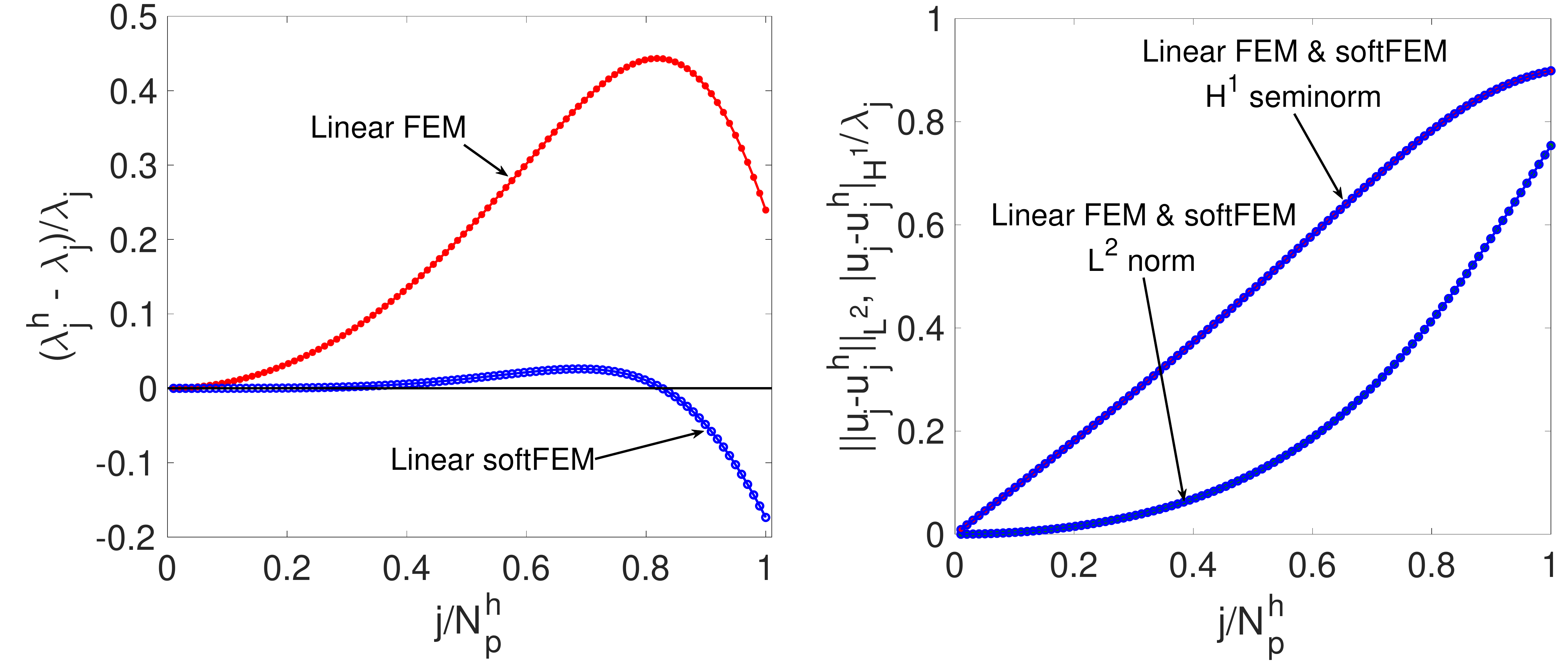} \\
\includegraphics[height=5.2cm]{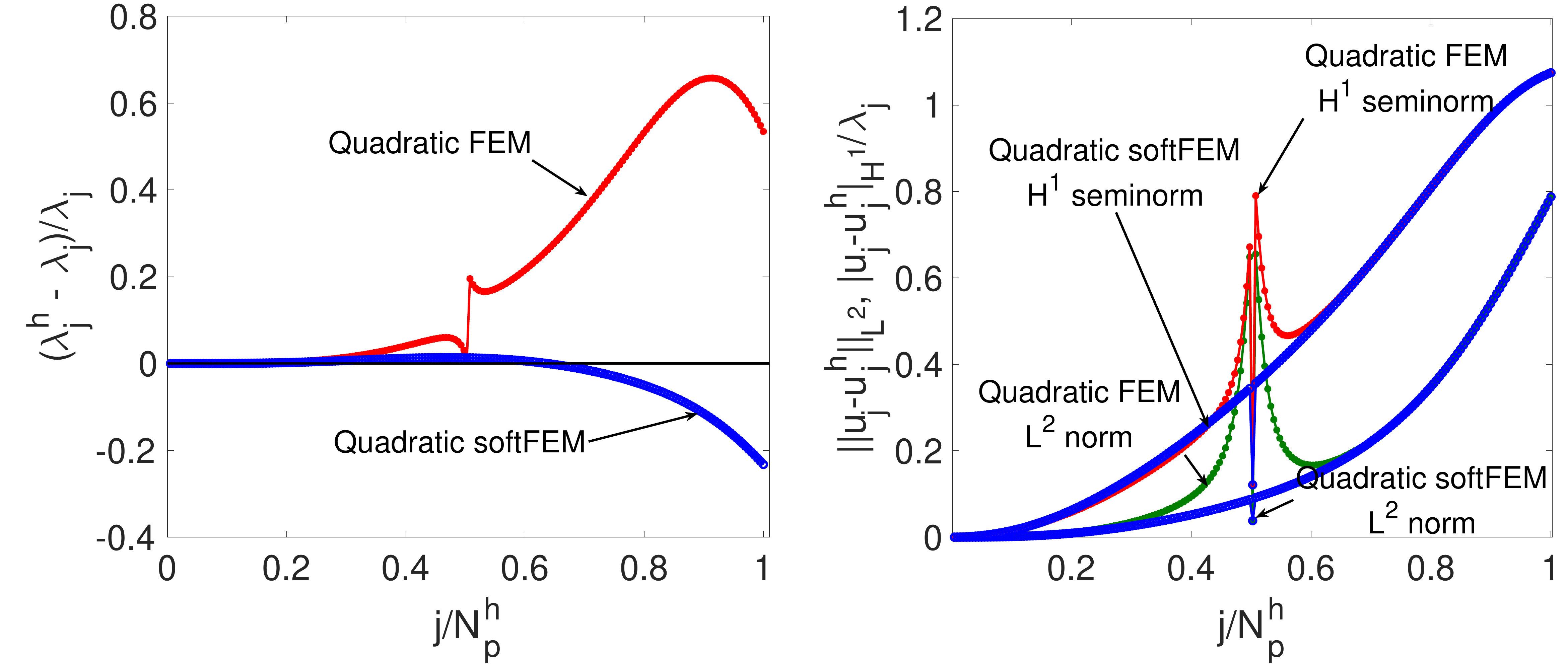} \\
\includegraphics[height=5.2cm]{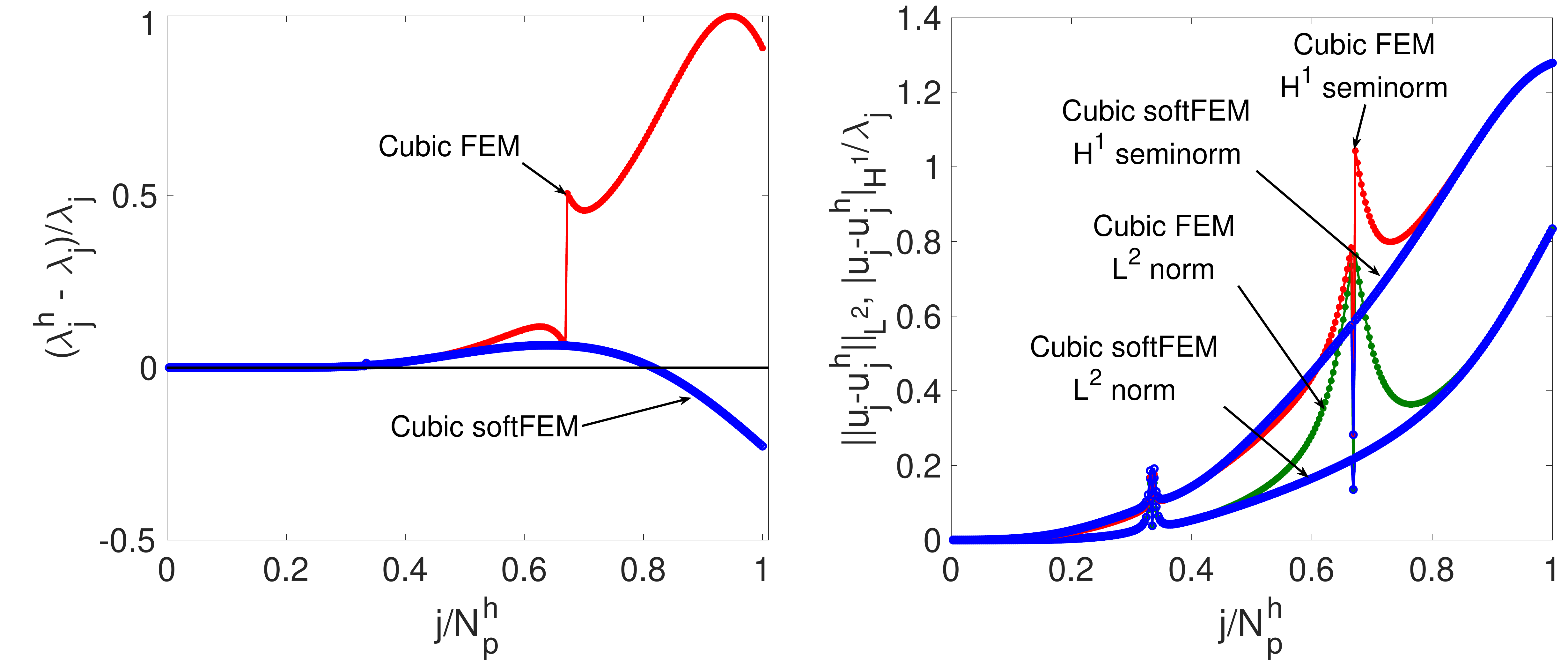}
\caption{Relative eigenvalue (left) and eigenfunction (right) errors for the 1D Laplace eigenvalue problem when using Galerkin FEM and softFEM with $N^h=100$ uniform elements and polynomial degrees $p\in\{1,2,3\}$. Upper row: $p=1$; middle row: $p=2$; bottom row: $p=3$. The eigenfunction errors for linear Galerkin FEM and linear softFEM are the same as both discretization methods give the same eigenvectors (but not the same eigenvalues).}
\label{fig:fem1D}
\end{figure}

To give the reader a first view on the benefits of softFEM over Galerkin FEM, 
we present in Figure~\ref{fig:fem1D} the relative eigenvalue and eigenfunction errors
for the 1D Laplace eigenvalue problem (with Dirichlet boundary conditions) using Galerkin FEM and softFEM{. W}e use a uniform mesh composed of $N^h=100$ elements and a polynomial degree $p\in\{1,2,3\}${. T}he total number of discrete eigenpairs is $N_p^h:=pN^h-1$. The benefit of using softFEM is evident when looking at the upper part of the spectrum. Another salient advantage of softFEM with respect to Galerkin FEM is that softFEM tempers the condition number of the stiffness matrix. This can have practically important consequences in the context of explicit time-marching schemes for time-dependent PDEs by reducing the CFL constraint on the time step. In many situations we observe that the stiffness reduction ratio scales linearly with $p$ and is of the order of $1+\frac{p}{2}$. 

The main mathematical results of this work can be summarized as follows. In Theorem~\ref{thm:coe} we show that in order to maintain the coercivity of the softFEM bilinear form $\hat a(\cdot, \cdot)$, the softness parameter can be chosen so that $\eta\in [0,\eta_{\max})$, where the limit value depends on the polynomial degree $p$ and the type of mesh (tensor-product or simplicial). Specifically $\eta_{\max}=\frac{1}{2p(p+1)}$ on tensor-product meshes and $\eta_{\max}=\frac{1}{2p(p+d-1)}$ on simplicial meshes (here $d\ge2$ denotes the space dimension). This result is established by means {of discrete} trace inequalities with sharp constants. In Theorem~\ref{thm:eveferr} we establish that softFEM maintains the same optimal convergence rates as Galerkin FEM. In Theorem~\ref{thm:p1} we prove for the 1D Laplace eigenvalue problem approximated by linear softFEM (i.e., $p=1$), that the choice $\eta=\frac{1}{2(p+1)(p+2)}=\frac{1}{12}$ leads to superconvergence of the eigenvalue errors (quartic convergence rate instead of quadratic). We retain this choice for the value of the softness parameter in the rest of this work and notice that it is compatible with the maximum value $\eta_{\max}$ obtained in Theorem~\ref{thm:coe}. Finally, in Theorem~\ref{thm:bds} we establish lower and upper bounds on the discrete softFEM eigenvalues by those approximated by Galerkin FEM. In particular, the lower bound shows that the optimal value for the stiffness reduction ratio should be 
$1+\frac{p}{2}$ on tensor-product meshes and $1+\frac{p}{4-d}$ on simplicial meshes with $d\in\{2,3\}$. Both values are close to those observed in our numerical experiments.   

The rest of this paper is organized as follows.  Section~\ref{sec:mi} presents the exact spectral problem, its Galerkin FEM discretization, the softFEM approximation, as well as the following salient results concerning softFEM: coercivity (Theorem~\ref{thm:coe}), error estimates (Theorem~\ref{thm:eveferr}), and lower and upper bounds on the discrete eigenvalues (Theorem~\ref{thm:bds}). Theorem~\ref{thm:eveferr} and Theorem~\ref{thm:bds} are proved in Section~\ref{sec:mi}, but the proof of Theorem~\ref{thm:coe} is postponed to Section~\ref{sec:ana}.
Section \ref{sec:1d} is concerned with the softFEM approximation of the 1D Laplace eigenvalue problem on uniform meshes. It contains the superconvergence result for softFEM (Theorem~\ref{thm:p1}) motivating the choice $\eta=\frac{1}{2(p+1)(p+2)}$ for the softness parameter, and numerical experiments for various polynomial degrees illustrating the benefits of using softFEM with respect to Galerkin FEM both for the accuracy of the upper part of the spectrum and for the stiffness reduction.  
Section~\ref{sec:num} collects more challenging numerical examples (Laplace eigenvalue problem in multiple dimensions, elliptic eigenvalue problem and non-uniform meshes for the 1D Laplace eigenvalue problem, and the use of simplicial meshes on the unit square and the L-shaped domain still for the Laplace eigenvalue problem) which corroborate the positive conclusions drawn on softFEM in Section \ref{sec:1d}.
In Section~\ref{sec:ana}, we first study discrete trace inequalities with sharp constants and then use these inequalities to prove Theorem~\ref{thm:coe}.
Concluding remarks are presented in Section~\ref{sec:conclusion}.

\section{Main idea and results} \label{sec:mi}

In this section{,} we state the elliptic eigenvalue problem and describe its approximation by
means of Galerkin FEM and softFEM. We then state the main results concerning softFEM.

\subsection{Problem statement}

Let $\Omega$ be a bounded, open subset of $\mathbb{R}^d$, $d\ge1$, 
with Lipschitz boundary $\partial \Omega$. For simplicity, we assume in what follows that
$\Omega$ is a polyhedron. We use standard notation for the Lebesgue and
Sobolev spaces. For any measurable subset $S\subseteq \Omega$, 
we denote the $L^2$-inner product and norm as $(\cdot,\cdot)_S$ and $\| \cdot \|_S$,
respectively, and the same notation is used for vector-valued fields. 
For any integer $m\ge1$,
we denote the $H^m$-norm and $H^m$-seminorm as $\| \cdot \|_{H^m(S)}$ and 
$| \cdot |_{H^m(S)}$, respectively. 

We consider the following second-order elliptic eigenvalue problem with homogeneous Dirichlet boundary conditions: Find an eigenpair $(\lambda, u)\in \mathbb{R}^+\times H^1_0(\Omega)$ such that $\|u\|_\Omega=1$ and
\begin{equation} \label{eq:pde}
\begin{alignedat}{2}
- \nabla \cdot (\kappa \nabla u)  & =  \lambda u &\quad &\text{in $\Omega$}, \\
u & = 0 &\quad &\text{on $\partial \Omega$},
\end{alignedat}
\end{equation}
with the diffusion coefficient $\kappa\in L^\infty(\Omega)$ uniformly bounded from below away from zero, and we set $\kappa_{\min}:=\text{ess\,inf}_{x\in\Omega}\kappa(x)>0$.
For $\kappa =1$, the problem~\eqref{eq:pde} reduces to the Laplace (Dirichlet) eigenvalue problem. The variational formulation of~\eqref{eq:pde} is 
\begin{equation} \label{eq:vf}
a(u, w) =  \lambda b(u, w), \quad \forall w \in H^1_0(\Omega), 
\end{equation}
with the bilinear forms
\begin{equation}
a(v,w) := (\kappa \nabla v, \nabla w)_\Omega,
\qquad
b(v,w) := (v,w)_\Omega.
\end{equation}
The eigenvalue problem~\eqref{eq:pde} has a countable set of eigenvalues $\lambda_j \in \mathbb{R}^+$ (see, for example, \cite[Sec. 9.8]{Brezis:11}) 
\begin{equation*}
0 < \lambda_1 < \lambda_2 \leq \lambda_3 \leq \ldots
\end{equation*}
and an associated set of $L^2$-orthonormal eigenfunctions $u_j$, that is,
$(u_j, u_k)_{{\Omega}} = \delta_{jk}$, 
where $\delta_{jk}$ is the Kronecker delta. 
With \eqref{eq:vf} in mind, the normalized eigenfunctions are also orthogonal in the energy inner product since we have
$a(u_j, u_k) =  \lambda_j b(u_j, u_k) = \lambda_j \delta_{jk}.$
In what follows we always sort the eigenvalues in ascending order counted with their
order of algebraic multiplicity.

\subsection{Galerkin FEM}

Let $(\mathcal{T}_h)_{h>0}$ be a shape-regular sequence of meshes of $\Omega$. A generic mesh element is denoted $\tau$, its diameter $h_\tau$, and its outward unit normal $\bfs{n}_\tau$. We set $h := \max_{\tau\in \mathcal{T}_h} h_\tau$. 
To stay general, we consider both tensor-product meshes where the mesh elements are cuboids (and so is the domain $\Omega$), and simplicial meshes where the mesh elements are simplices (triangles if $d=2$, tetrahedra if $d=3$).
Let $p\ge1$ be the polynomial degree. 
Let $\polP_p(\tau)$ (resp., $\polQ_p(\tau)$) 
be the space composed of the restriction to $\tau$ of polynomials
of total degree at most $p$ (resp., of degree at most $p$ in each variable).
For tensor-product meshes, the Galerkin finite element approximation space is defined as 
\begin{equation}
V^h_p := \{ v_h \in C^0(\overline \Omega): v_h |_{\partial \Omega} = 0, \forall \tau\in \mathcal{T}_h, v_h|_\tau \in \mathbb{Q}_p(\tau) \},
\end{equation}
whereas for simplicial meshes, it is defined as
\begin{equation}
V^h_p := \{ v_h \in C^0(\overline \Omega): v_h |_{\partial \Omega} = 0, \forall \tau\in \mathcal{T}_h, v_h|_\tau \in \mathbb{P}_p(\tau) \}.
\end{equation}
It is well-known that in both cases $V^h_p \subset H^1_0(\Omega)$. 

The Galerkin FEM approximation of~\eqref{eq:pde} seeks $(\lambda^h,u^h) \in \mathbb{R}^+\times V_p^h$ such that $\|u^h\|_\Omega=1$ and
\begin{equation} \label{eq:vfh}
a(u^h, w^h) =  \lambda^h b(u^h, w^h), \quad \forall w^h \in V^h_p.
\end{equation} 
The algebraic realization of~\eqref{eq:vfh} follows by choosing basis functions
$\{\phi_j^h\}_{j\in\{1,\ldots, N^h_p\}}$ of $V_p^h$ with $N^h_p:=\text{dim}(V_p^h)$ (typically,
one considers nodal basis functions). This 
leads to the following generalized matrix eigenvalue problem (GMEVP):
\begin{equation} \label{eq:mevp}
\mathbf{K} \mathbf{U} = \lambda^h \mathbf{M} \mathbf{U},
\end{equation}
where $\mathbf{K}_{kl} :=  a(\phi_k^h, \phi_l^h)$ and $\mathbf{M}_{kl} := b(\phi_k^h, \phi_l^h),$ for all $k,l\in \{1,\ldots,N^h_p\}$, 
are the entries of the stiffness and mass matrices, respectively, 
and $\mathbf{U}\in \mathbb{R}^{N^h_p}$ is the eigenvector collecting 
the components of $u^h$ in the chosen basis.

\subsection{SoftFEM}  \label{sec:softFEM}

For all $\tau\in\mathcal{T}_h$, we define $h_\tau^0$ to be the length of the smallest edge of $\tau$ if $\tau$ is a cuboid, whereas we set $h_\tau^0:=\frac{d|\tau|}{|\partial\tau|}$ if $\tau$ is a simplex. Let $\mathcal{F}_h^i$ be the collection of the mesh interfaces. For all $F\in \mathcal{F}_h^i$, we have $F=\partial\tau_1\cap \partial\tau_2$ for two distinct mesh elements $\tau_1,\tau_2\in\mathcal{T}_h$. We then set
\begin{equation}\label{eq:def_F_based}
h_F:=\min(h_{\tau_1}^0,h_{\tau_2}^0), \qquad
\kappa_F:=\min(\kappa_{\tau_1},\kappa_{\tau_2}),
\end{equation}
with $\kappa_\tau:=\text{ess\,inf}_{x\in\tau}\kappa(x)$ (i.e., $\kappa_F$ is the smallest value of $\kappa$ on the two elements that share the interface $F$).
Moreover, for any function $v^h\in V_p^h$, we define the jump of its normal derivative across $F$ as
\begin{equation}
\lsem \nabla v^h \cdot \bfs{n} \rsem_F := \nabla v^h |_{\tau_1} \cdot \bfs{n}_{\tau_1} + \nabla v^h |_{\tau_2} \cdot \bfs{n}_{\tau_2}.
\end{equation}
We drop the subscript $F$ when the context is unambiguous.

The softFEM approximation of~\eqref{eq:pde} seeks $(\hat \lambda^h ,\hat u^h)\in \mathbb{R}^+\times V^h_p$ such that $\|\hat u^h\|_\Omega=1$ and
\begin{equation} \label{eq:softFEM}
\hat a(\hat u^h, w^h) =  \hat\lambda^h b(\hat u^h, w^h), \quad \forall w^h \in V^h_p,
\end{equation} 
where for all $v^h,w^h \in V^h_p$,
\begin{equation} \label{eq:softFEMbf}
\hat a(\cdot, \cdot) := a(\cdot, \cdot)  - \eta s(\cdot, \cdot) \qquad \text{with} \qquad s(v^h,w^h)  := \sum_{F \in \mathcal{F}_h^i} \kappa_F h_F (\lsem \nabla v^h \cdot \bfs{n} \rsem, \lsem\nabla w^h \cdot \bfs{n} \rsem )_F, 
\end{equation}
and $\eta \ge 0$ is a parameter to be specified below.
The terminology \textit{softFEM} is motivated by the fact that the term $- \eta s(\cdot, \cdot)$ reduces the stiffness of the system.
We refer to  $\eta$ as the \textit{softness parameter}.
We will see below that one can take $\eta\in [0,\eta_{\max})$ for some
$\eta_{\max}$ depending on the polynomial degree $p$ and the type of mesh elements so
that the bilinear form $\hat a(\cdot,\cdot)$ remains coercive. 
When $\eta=0,$ softFEM reduces to FEM.

Similarly to Galerkin FEM, the algebraic realization of the
softFEM approximation \eqref{eq:softFEM}  leads to the GMEVP
\begin{equation} \label{eq:npmevp}
\hat{\mathbf{K}} \hat{\mathbf{U}} = \hat \lambda^h \mathbf{M} \hat{\mathbf{U}},
\end{equation}
where $\hat{\mathbf{K}} := \mathbf{K} - \eta \mathbf{S}$ with $\mathbf{S}_{kl} :=  s(\phi_k^h, \phi_l^h)$, $\mathbf{K}$ and $\mathbf{M}$ are respectively the stiffness and mass matrices as in \eqref{eq:mevp}, and $\hat{\mathbf{U}}$ is the eigenvector collecting the components of $\hat u^h$ in the chosen basis $\{\phi_j^h\}_{j\in\{1,\ldots,N^h_p\}}$ of $V_p^h$.

\begin{remark}[Variants]
For $p\ge2$, the stiffness can be further reduced by imposing least-squares penalties on higher-order derivative jumps. However, these additional terms increase the computational costs while our numerical experiments (not shown for brevity) indicate only a further marginal improvement in terms of spectral errors. 
We also mention the recent work \cite{deng2021boundary} which penalizes both the higher-order derivatives as well as the mass bilinear form near the boundary to eliminate the so-called outliers in isogeometric spectral approximations.  
\end{remark}

\subsection{Main results on softFEM} \label{sec:main_results}

In this section{,} we present our main results on softFEM. We first derive an upper bound
on the softness parameter to ensure coercivity of the bilinear form $\hat a(\cdot,\cdot)$.
To improve readability, the proof is postponed to Section~\ref{sec:ana}.

\begin{theorem}[Coercivity]\label{thm:coe}
Let $\hat a(\cdot, \cdot)$ be defined in \eqref{eq:softFEMbf}. 
Set $\eta_{\max}:=\frac{1}{2p(p+1)}$ for tensor-product meshes with $d\ge1$ and
$\eta_{\max}:=\frac{1}{2p(p+d-1)}$ for simplicial meshes with $d\ge2$. Assume that
the softness parameter $\eta\in [0,\eta_{\max})$. The following holds: 
\begin{equation}
\beta_1 | w^h |^2_{H^1(\Omega)} \le \hat a(w^h, w^h), \qquad \forall w^h \in V^h_p, 
\end{equation} 
with $\beta_1:=\kappa_{\min}(1-\frac{\eta}{\eta_{\max}}) >0$.
\end{theorem}


Let us now consider the convergence of eigenvalues and eigenfunctions for softFEM. 
{
We define the solution operator $T:L^2(\Omega)\to H^1_0(\Omega)\subset L^2(\Omega)$ such that for all $\phi\in L^2(\Omega)$,
\begin{equation}
a(T(\phi),w)=b(\phi,w), \qquad \forall w\in H^1_0(\Omega).
\end{equation}
We n}otice that $T$ is selfadjoint and compact, and the elliptic regularity theory implies that there is $s\in (\frac12,1]$ such that $T$ maps boundedly from $L^2(\Omega)$ into $H^{1+s}(\Omega)$. Moreover $(\lambda,u)$ is an eigenpair of~\eqref{eq:vf} if and only if $(\mu,u)$ is an eigenpair of $T$ with $\mu=\lambda^{-1}$. 

\begin{theorem}[Eigenvalue and eigenfunction errors]\label{thm:eveferr}
Let $(\lambda_j, u_j) \in \mathbb{R}^+\times H^1_0(\Omega)$ solve \eqref{eq:vf} and let $(\hat{\lambda}_j^h, \hat u_j^h) \in \mathbb{R}^+\times V^h_p$ solve \eqref{eq:softFEM} with the normalizations $\| u_j \|_{\Omega} = 1$ and $ \| \hat u_j^h \|_{\Omega} = 1$. Let $s\in (\frac12,1]$ be the index of elliptic regularity. Assume that there is $t\in [s,p]$ and a constant $C_t$ such that one has the following smoothness property: $\|\phi\|_{H^{1+t}(\Omega)} + \|T(\phi)\|_{H^{1+t}(\Omega)} \le C_t\|\phi\|_\Omega$ for all $\phi \in G_j:=\text{ker}(\mu_jI-T)$ with $\mu_j:=\lambda_j^{-1}$. Then, the following holds:
\begin{equation}
\big| \hat \lambda_j^h - \lambda_j \big|  \le Ch^{2t}, \qquad | u_j - \hat u_j^h |_{H^1(\Omega)} \le Ch^{t}, 
\end{equation}
where $C$ is a positive constant independent of the mesh-size $h$. The convergence rates are optimal whenever $t=p$.
\end{theorem}

\begin{proof}
We cannot apply directly the classical theory for error analysis derived in \cite[Thm.~7.2 \& 7.4]{babuvska1991eigenvalue} since the softFEM bilinear form $\hat a(\cdot, \cdot)$ differs from $a(\cdot, \cdot)$. Instead, we can apply the extension of this theory presented in \cite[Chap.~48]{Ern_Guermond_FEs_II_2021} to finite element approximations with so-called variational crimes. We can work on the extended space $Y^h:=V_p^h+H^{1+s}(\Omega)$ and establish the boundedness of $\hat a$ on $Y^h\times Y^h$ using the $H^1$-seminorm augmented by $s(\cdot,\cdot)^{\frac12}$. Optimal approximation properties in this norm are readily derived for smooth functions. Moreover, consistency holds true since we have $s(u_j,y)=0$ for all $y\in Y^h$ because $s>\frac12$. This implies the above error estimates.
\end{proof}

\begin{remark}[Pythagorean identity]
A classical identity relating the eigenvalue and eigenfunction errors (see, e.g., 
\cite[Chap.~6]{strang1973analysis}) is
\begin{equation*} 
\| u_j - \hat u_j^h \|_E^2 = \lambda_j \| u_j - \hat u_j^h \|^2_{\Omega} + \hat{\lambda}_j^h - \lambda_j,
\end{equation*}
where $\| \cdot \|_{E}^2 := \hat a(\cdot, \cdot)\ge \beta_1|\cdot|_{H^1(\Omega)}^2$ 
owing to Lemma~\ref{thm:coe}.
\end{remark}

Our third main result quantifies the
stiffness reduction by softFEM for one particular choice of the softness parameter
$\eta$ that is further motivated in Section~\ref{sec:1d} (see, in particular, Theorem~\ref{thm:p1}), namely $\eta=\frac{1}{2(p+1)(p+2)}$. Notice that $\eta<\eta_{\max}=\frac{1}{2p(p+1)}$ for tensor-product meshes and that $\eta<\eta_{\max}=\frac{1}{2p(p+d-1)}$ with $d\in\{2,3\}$ on simplicial meshes.

\begin{theorem}[Eigenvalue lower and upper bounds]\label{thm:bds}
Assume that $\eta=\frac{1}{2(p+1)(p+2)}$. Assume that $d\in\{2,3\}$ if simplicial meshes are used.
Let $j\in\mathbb{N}$, let $(\lambda_j^h, u_j^h) \in \mathbb{R}\times V^h_p$ solve \eqref{eq:vfh}, and let $(\hat{\lambda}_j^h, \hat u_j^h) \in \mathbb{R}\times V^h_p$ solve \eqref{eq:softFEM} with the normalizations $\| u_j^h \|_{\Omega} = 1$ and $ \| \hat u_j^h \|_{\Omega} = 1$.  The following holds:
\begin{equation}\label{eq:bds}
\gamma_p \lambda_j^h \le \hat \lambda_j^h <  \lambda_j^h,
\end{equation}
with $\gamma_p:=\frac{2}{p+2}$ on tensor-product meshes and 
$\gamma_p:=\frac{4-d}{p+4-d}$ on simplicial meshes.
\end{theorem}

\begin{proof}
For all $v^h\in V_p^h\setminus\{0\}$, let us define the Rayleigh quotients
\[
R(v^h) := \frac{a(v^h,v^h)}{b(v^h,v^h)}, \qquad
\hat R(v^h) := \frac{\hat a(v^h,v^h)}{b(v^h,v^h)}{.}
\]
As shown in Section~\ref{sec:ana} (see \eqref{eq:hat_a_a}), we have
\[
(1-2p(p+1)\eta)a(v^h,v^h)\le \hat a(v^h,v^h) < a(v^h,v^h)
\]
on tensor-product meshes and 
\[
(1-2p(p+d-1)\eta)a(v^h,v^h)\le \hat a(v^h,v^h) < a(v^h,v^h)
\]
on simplicial meshes.
With the choice $\eta = \frac{1}{2(p+1)(p+2)}$, a direct calculation shows that
\[
\gamma_p a(v^h,v^h)\le \hat a(v^h,v^h) < a(v^h,v^h),
\] 
with $\gamma_p$ defined in the assertion, which readily implies that 
\begin{equation} \label{eq:comp_R}
\gamma_p R(v^h) \le \hat R(v^h) < R(v^h).
\end{equation}
Let $V_j$ denote the set of the subspaces of $V_p^h$ of dimension $j\ge1$.  
Classical results on the Rayleigh quotient imply that
\begin{align*}
\lambda_j^h = \min_{E_j\in V_j} \max_{v^h\in E_j} R(v^h),  
\qquad
\hat \lambda_j^h = \min_{E_j\in V_j} \max_{v^h\in E_j} \hat R(v^h).
\end{align*}
The bounds in \eqref{eq:bds} then readily follow from~\eqref{eq:comp_R}. 
\end{proof}

Since the stiffness matrices $\mathbf{K}$ and $\hat{\mathbf{K}}$ are symmetric,
their condition numbers are given by
\begin{equation}
\sigma := \frac{\lambda^h_{\max}}{\lambda^h_{\min}}, \qquad \hat \sigma := \frac{\hat \lambda^h_{\max}}{\hat \lambda^h_{\min}},
\end{equation}
where $\lambda^h_{\max}, \hat \lambda^h_{\max}$ are the largest eigenvalues and $\lambda^h_{\min}, \hat \lambda^h_{\min}$ are the smallest eigenvalue of the GMEVPs \eqref{eq:mevp} and \eqref{eq:npmevp} that are associated with \eqref{eq:vfh} and \eqref{eq:softFEM}, respectively. 
We define the \textit{stiffness reduction ratio} of softFEM with respect to Galerkin FEM as \begin{equation} \label{eq:srr}
\rho := \frac{\sigma}{\hat\sigma} = \frac{\lambda^h_{\max} }{\hat \lambda^h_{\max} } \cdot \frac{\hat \lambda^h_{\min} }{\lambda^h_{\min}}.
\end{equation}
In general, for Galerkin FEM and softFEM with sufficient elements (i.e., as $h\to0$), one has $\lambda^h_{\min} \approx \hat \lambda^h_{\min}$. Thus, the stiffness reduction ratio depends only on the largest eigenvalues for both methods. Since softFEM leads to a smaller largest eigenvalue, softFEM lowers the condition number of the stiffness matrix, i.e., $\rho\ge1$. 
We define the \textit{asymptotic stiffness reduction ratio} of softFEM with respect to Galerkin FEM as 
\begin{equation} \label{eq:asrr}
\rho_\infty := \lim_{h\to 0} \frac{\lambda^h_{\max}}{\hat \lambda^h_{\max}}.
\end{equation}
Theorem~\ref{thm:bds} shows that for $\eta=\frac{1}{2(p+1)(p+2)}$, the best possible asymptotic stiffness reduction ratio is $1+\frac{p}{2}$ on tensor-product meshes and $1+\frac{p}{4-d}$ on simplicial meshes with $d\in\{2,3\}$. Notice that for both types of meshes, this value grows linearly with $p$. Our numerical experiments reported in Section~\ref{sec:num_1D} for the 1D Laplace eigenvalue problem show that the asymptotic stiffness reduction ratio is indeed $\rho_\infty=1+\frac{p}{2}$. Moreover, the values of the asymptotic stiffness reduction ratio observed in the more general situations studied in Section~\ref{sec:num} are also close to the predictions of Theorem~\ref{thm:bds}.
Finally, we define the \textit{stiffness reduction percentage} of softFEM with 
respect to Galerkin FEM as
\begin{equation}
\varrho = 100 \frac{\sigma-\hat \sigma}{\sigma}\,\% = 100(1-\rho^{-1}) \, \%,
\end{equation} 
and the \textit{asymptotic stiffness reduction percentage} as 
$\varrho_\infty := 100 (1-\rho_\infty^{-1})\, \%$, respectively. 

\begin{remark}[SoftFEM eigenvalues]
It is well-known that for Galerkin FEM, one has $ \lambda_j \le \lambda_j^h$ for all $j\ge1$, but this is not necessarily the case for softFEM. Our numerical experiments indicate that softFEM approximates the exact eigenvalues from above in the low-frequency region and from below in the high-frequency region. 
\end{remark}

\section{Laplace eigenvalue problem in 1D} \label{sec:1d}
In this section, we focus on the spectral problem \eqref{eq:pde} with $\Omega := (0, 1)$ and $\kappa := 1$, that is, on the 1D Laplace eigenvalue problem. In this case, the problem~\eqref{eq:pde} has exact eigenvalues and $L^2$-normalized eigenfunctions 
\begin{equation} \label{eq:p1d}
\lambda_j = j^2 \pi^2 \quad \text{and} \quad u_j(x) = \sqrt{2} \sin( j\pi x), \quad j = 1, 2, \ldots,
\end{equation}
respectively. We partition the interval $\Omega = (0,1)$ into $N^h$ uniform elements so that the mesh size $h = 1/N^h.$ We first focus on the case of linear finite elements ($p=1$)
and derive some analytical results showing that in this case the optimal choice for the 
softness parameter is $\eta=\frac{1}{12}$, that is, $\eta=\frac{1}{2(p+1)(p+2)}$ for $p=1$. 
Then we present numerical experiments for this choice of the softness parameter and various polynomial degrees.

\subsection{Analytical results for linear softFEM}

The advantage of using linear elements is that it is possible to compute analytically
the eigenvalues and eigenvectors for Galerkin FEM and softFEM. 
Firstly, it is well-known that the bilinear forms $a(\cdot, \cdot)$ and $b(\cdot, \cdot)$ with $p=1$ lead to the following stiffness and mass matrices:
\begin{equation} \label{eq:km1dp1}
\mathbf{K} = 
\frac{1}{h}
\begin{bmatrix}
2 & -1 &  \\
-1 & 2 & -1 &  \\
   & \ddots & \ddots & \ddots &  \\
  &  & -1 & 2 & -1  \\
&& & -1 & 2  \\
\end{bmatrix}, \qquad
\mathbf{M} = h
\begin{bmatrix}
 \frac{2}{3} & \frac{1}{6} \\[0.2cm]
\frac{1}{6} & \frac{2}{3} & \frac{1}{6} \\[0.2cm]
 & \ddots & \ddots & \ddots &  \\[0.2cm]
&  & \frac{1}{6} & \frac{2}{3} & \frac{1}{6}  \\[0.2cm]
&& & \frac{1}{6} & \frac{2}{3}  \\
\end{bmatrix},
\end{equation}
which are of order ${(N^h-1) \times (N^h-1)}$. The bilinear form $s(\cdot, \cdot)$ leads to the matrix
\begin{equation} \label{eq:j1dp1}
\mathbf{S} = 
\frac{1}{h}
\begin{bmatrix}
5 & -4 & 1 \\
-4 & 6 & -4 & 1  \\
1 & -4 & 6 & -4 & 1 \\
   & \ddots & \ddots & \ddots &  \ddots & \ddots \\
  &  & 1 & -4 & 6 & -4 & 1  \\
&& & 1 & -4 & 6 & -4  \\
&& & & 1 & -4 & 5 \\
\end{bmatrix}, 
\end{equation}
which is also of order ${(N^h-1) \times (N^h-1)}$. 
Recall that we then have $\hat{\mathbf{K}} := \mathbf{K} - \eta \mathbf{S},$
and that according to Theorem~\ref{thm:coe}, we must take the softness parameter 
$\eta\in [0,\eta_{\max})$ with $\eta_{\max}=\frac{1}{2p(p+1)}=\frac16$ since $p=1$ here.

\begin{lemma}[Analytical eigenvalues and eigenvectors] \label{lem:ep1d}
\textup{(i)} Galerkin FEM approximation: 
The GMEVP $\mathbf{K} \mathbf{U} = \lambda^h \mathbf{M} \mathbf{U}$ has eigenpairs
$(\lambda_j^h, \mathbf{U}_j)$ for all $j\in\{1,\ldots,N^h-1\}$ with 
\begin{equation} \label{eq:p1ev0}
\lambda_j^h = \frac{6}{h^2}\frac{1 - \cos(t_j)}{2 + \cos(t_j)}, \qquad \mathbf{U}_{j} = c_j \big(\sin(kt_j)\big)_{k \in\{1,\ldots, N^h-1\}},
\end{equation}
with $t_j:=j \pi h$ and some normalization constant $c_j> 0$.
\textup{(ii)} SoftFEM approximation: The GMEVP $\hat{\mathbf{K}} \hat{\mathbf{U}} = \hat \lambda^h \mathbf{M} \hat{\mathbf{U}}$ has eigenpairs
$(\hat\lambda_j^h, \hat{\mathbf{U}}_j)$ for all $j\in\{1,\ldots,N^h-1\}$ with 
\begin{equation} \label{eq:p1ev}
\hat\lambda_j^h = \frac{6}{h^2}\frac{1 + 3\eta - (1+4\eta) \cos(t_j) + \eta \cos(2t_j)}{2 + \cos(t_j)}, \qquad \hat{\mathbf{U}}_{j} =\mathbf{U}_{j}.
\end{equation}
\end{lemma}
\begin{proof}
The result \eqref{eq:p1ev0} is well-known; see, for example, \cite[Sec. 2]{boffi2010finite} or \cite[Sec. 4]{hughes2008duality}, whereas the result \eqref{eq:p1ev} follows for instance from an application of \cite[Thm.~2.1]{deng2021analytical}.
\end{proof}

An interesting consequence of \eqref{eq:p1ev0}-\eqref{eq:p1ev} is that for linear softFEM,
the stiffness reduction ratio and the asymptotic stiffness reduction ratio are 
\begin{equation} 
\rho = \frac{\lambda^h_{\max} }{\hat \lambda^h_{\max} } \cdot \frac{\hat \lambda^h_{\min} }{\lambda^h_{\min}}= \frac{5 + \cos(\pi h) }{5 - \cos(\pi h) },
\qquad
\rho_\infty = \lim_{h \to 0} \frac{5 + \cos(\pi h) }{5 - \cos(\pi h) }  = \frac32.
\end{equation}
Thus, asymptotically, linear softFEM reduces the stiffness of Galerkin FEM by about $33.3\%$. 

For all $\eta\in [0,\eta_{\max})$ with $\eta_{\max}=\frac{1}{2p(p+1)}=\frac16$,
Theorem~\ref{thm:eveferr} shows that one should expect a quadratic convergence rate
for the discrete eigenvalues. We now show that for the specific choice
$\eta=\frac{1}{2(p+1)(p+2)}=\frac{1}{12}$, one obtains a quartic convergence rate, uniformly for all the discrete eigenvalues. 

\begin{theorem}[Eigenvalue superconvergence] \label{thm:p1}
Let $\lambda_j$ be the $j$-th exact eigenvalue of \eqref{eq:pde} and let $\hat \lambda_j^h$ be the $j$-th approximate eigenvalue using linear softFEM. Assume that $\eta = \frac{1}{2(p+1)(p+2)} = \frac{1}{12}$. The following holds:
\begin{equation}
\frac{ |\hat \lambda_j^h - \lambda_j|}{\lambda_j}  < \frac{1}{360} (j \pi h)^4, \qquad \forall j\in\{1,\ldots, N^h-1\}.
\end{equation}
\end{theorem}

\begin{proof}
The exact eigenvalues $\lambda_j$ are given in \eqref{eq:p1d}, and the approximate eigenvalues $\hat \lambda_j^h$ are given in \eqref{eq:p1ev}. 
To motivate the result of Theorem~\ref{thm:p1}, we observe that 
applying a Taylor expansion to $\hat \lambda_j^h$, we obtain (recall that $t_j:=j\pi h$)
\begin{equation*}
\frac{ \hat \lambda_j^h - \lambda_j}{ \lambda_j} = \frac{1-12\eta}{12} t_j^2 + \frac{1}{360}t_j^4 - \frac{17- 84\eta}{60480}t_j^6+ \mathcal{O}(t_j^8),
\end{equation*} 
showing that the choice $\eta=\frac{1}{12}$ leads to a cancellation of the dominant term in the expansion and that the sixth-order term has a negative coefficient. More rigorously, using \eqref{eq:p1d}, \eqref{eq:p1ev}, and algebraic manipulations, we infer that
\begin{equation*}
\frac{ | \hat \lambda_j^h - \lambda_j |}{ \lambda_j} = \left| \frac{9 - (2 + \cos(t_j))^2 }{t_j^2 (2 + \cos(t_j) )} -1 \right|.
\end{equation*} 
Since $t_j$ samples the interval $(0,\pi)$, we can consider a continuous variable $t\in (0,\pi)$ and prove more generally that 
\[
\left| \frac{9 - (2 + \cos(t))^2 }{t^2 (2 + \cos(t) )} -1 \right| < \frac{1}{360}t^4,
\]
or, equivalently, that 
$$ 
- t^6 (2 + \cos(t ) ) < 3240 - 360 (2 + \cos(t ))^2 - 360 t^2 (2 + \cos(t ) ) < t^6 (2 + \cos(t ) ), 
$$
for all $t\in (0,\pi)$. For the first inequality, we notice that the function 
$$
f(t) := t^6 (2 + \cos(t ) ) + 3240 - 360 (2 + \cos(t ))^2 - 360 t^2 (2 + \cos(t ) )
$$
is increasing on $(0,t_0)$ and decreasing on $(t_0, \pi)$ with $t_0\approx 2.79911$, that $f(0)=0$ and $f(\pi) \approx 0.800921$. The minimum value of $f$ in $(0,\pi)$ is thus $f(0)=0$.
For the second inequality, we notice that the function 
$$
g(t) := t^6 (2 + \cos(t ) ) - 3240 + 360 (2 + \cos(t ))^2 + 360 t^2 (2 + \cos(t ) )
$$
is increasing on $(0, \pi)$ and that $g(0)=0$. This completes the proof.
\end{proof}

\begin{remark}[{Literature}]
{The same value for the penalty parameter to achieve superconvergence
is obtained in \cite{burman2016linear} for the Helmholtz problem under
Robin boundary conditions, still for $p=1$. 
Values of the penalty parameter for 
$p\in\{2,3\}$ are derived in \cite{DuWu:15}.}
\end{remark}

\subsection{Numerical results for arbitrary-order softFEM in 1D}\label{sec:num_1D}

In this section, we explore numerically softFEM for various polynomial degrees $p\ge1$ using in
all cases the softness parameter $\eta=\frac{1}{2(p+1)(p+2)}$.

\begin{table}[ht]
\centering 
\begin{tabular}{| c | c || ccc | ccc | cc |}
\hline
$p$ & $N^h$ & $\frac{|\hat \lambda^h_1-\lambda_1|}{\lambda_1}$ &  $|u_1- \hat u^h_1|_{H^1}$ & $\| u_1- \hat u_1^h \|_{L^2}$ &  $\frac{|\hat \lambda^h_6-\lambda_6|}{\lambda_6}$ &  $|u_6 -\hat u^h_6 |_{H^1}$ & $\|u_6 - \hat u^h_6\|_{L^2}$ \\[0.1cm] \hline
 & 8 & 6.54e-5   & 3.58e-1   & 5.85e-3   & 2.10e-2   & 1.40e1   & 3.56e-1 \\[0.1cm]
 & 16 & 4.12e-6   & 1.78e-1   & 1.44e-3   & 4.80e-3   & 6.63   & 6.06e-2 \\[0.1cm]
1& 32 & 2.58e-7   & 8.91e-2   & 3.60e-4   & 3.27e-4   & 3.23   & 1.35e-2 \\[0.1cm]
 & 64 & 1.61e-8   & 4.45e-2   & 8.98e-5   & 2.08e-5   & 1.61   & 3.27e-3\\[0.1cm] 
 & rate & 4.00 &  1.00 & 2.01 & 3.38 & 1.04 & 2.25 \\[0.1cm] \hline
  & 4 & 4.38e-4  & 7.57e-2  & 2.54e-3 & 3.08e-2  & 1.37e1  & 2.82e-1 \\[0.1cm]
 & 8 & 3.15e-5  & 1.84e-2   &3.40e-4 & 1.11e-2   &3.95   &4.47e-2 \\[0.1cm]
 2& 16 & 2.04e-6 &  4.53e-3 &  4.33e-5 & 1.80e-3 &  1.04 &  7.78e-3  \\[0.1cm]
  & 32 & 1.29e-7   &1.13e-3   &5.43e-6 & 1.50e-4   &2.52e-1  & 1.11e-3 \\[0.1cm]
 & 64 & 8.06e-9  & 2.82e-4   &6.80e-7 & 1.02e-5   & 6.15e-2  & 1.45e-4 \\[0.1cm]
 & rate & 3.94 &  2.02 & 2.97 & 2.93 & 1.96 & 2.72 \\[0.1cm] \hline
  & 4 & 1.16e-7 &  5.82e-3  & 8.08e-5 &  4.32e-2 &  5.24 &  1.04e-1 \\[0.1cm]
  & 8 & 4.47e-10 &  7.19e-4 &  4.80e-6  & 7.64e-4 &  9.12e-1 &  9.29e-3 \\[0.1cm]
3  & 16 & 2.08e-12 &  8.96e-5 &  2.96e-7 &  3.02e-6 &  1.20e-1 &  4.41e-4 \\[0.1cm]
  & 32 & 4.04e-13 &  1.12e-5 &  1.85e-8 &  1.15e-8  & 1.46e-2  & 2.48e-5 \\[0.1cm]
& rate & 6.21 &  3.01 & 4.03 & 7.35 & 2.84 & 4.05 \\[0.1cm] \hline
  & 4 & 4.55e-9 &  2.71e-4 &  4.54e-6  & 2.29e-4 &  2.12 &  2.39e-2 \\[0.1cm]
4  & 8 & 2.09e-11 &  1.55e-5 &  1.47e-7 &  6.70e-6  & 1.38e-1 &  7.88e-4 \\[0.1cm]
  & 16 & 1.25e-13 &  9.38e-7 &  4.65e-9 &  9.01e-8 &  8.72e-3 &  3.24e-5 \\[0.1cm]
 & rate & 7.58 &  4.09 & 4.97 & 5.65 & 3.96 & 4.77 \\[0.1cm] \hline
\end{tabular}
\caption{Errors and convergence rates for the first and sixth eigenpairs using softFEM and polynomial degrees $p\in\{1,\ldots,4\}$.}
\label{tab:softFEM1d} 
\end{table}

Recall that Figure \ref{fig:fem1D} shows the relative eigenvalue and eigenfunction errors for Galerkin FEM and softFEM with $N^h=100$ uniform elements and polynomial orders $p\in\{1,2,3\}$. Notice that there are $N_p^h:=pN^h-1$ eigenpairs both for Galerkin FEM and for softFEM. We refer the reader to Section~\ref{sec:sb} for a brief discussion on the structure of the discrete spectrum for Galerkin FEM, including the notions of acoustic/optical branches and stopping bands. The improvement offered by softFEM over Galerkin FEM for the eigenvalues 
is clearly visible in Figure \ref{fig:fem1D} over the whole spectrum. For the eigenfunctions, there is no difference for $p=1$ (see Lemma~\ref{lem:ep1d}), whereas the improvement of softFEM over Galerkin FEM for $p\in\{2,3\}$ is salient around the stopping bands (that is, around $j=N^h$ for $p=2$ and around $j\in\{N^h,2N^h\}$ for $p=3$). Incidentally, we notice that for the $H^1$-seminorm, the errors in the low-frequency region are slightly larger with softFEM than with Galerkin FEM, although the convergence order for softFEM remains optimal. This is expected since in the low-frequency region, best-approximation errors in the finite element space decay optimally, and the softFEM approximation leads to an additional optimally-converging contribution due to the interface jump penalty on the normal gradient.
Table \ref{tab:softFEM1d} reports the errors for the first and sixth eigenpairs using softFEM and polynomial degrees $p\in\{1,\ldots,4\}$. We observe that in all the cases, the convergence rates match well the predictions of Theorem~\ref{thm:eveferr} {(and of Theorem~\ref{thm:p1} for $p=1$)}.

\begin{figure}[h!]
\centering
\includegraphics[height=5.5cm]{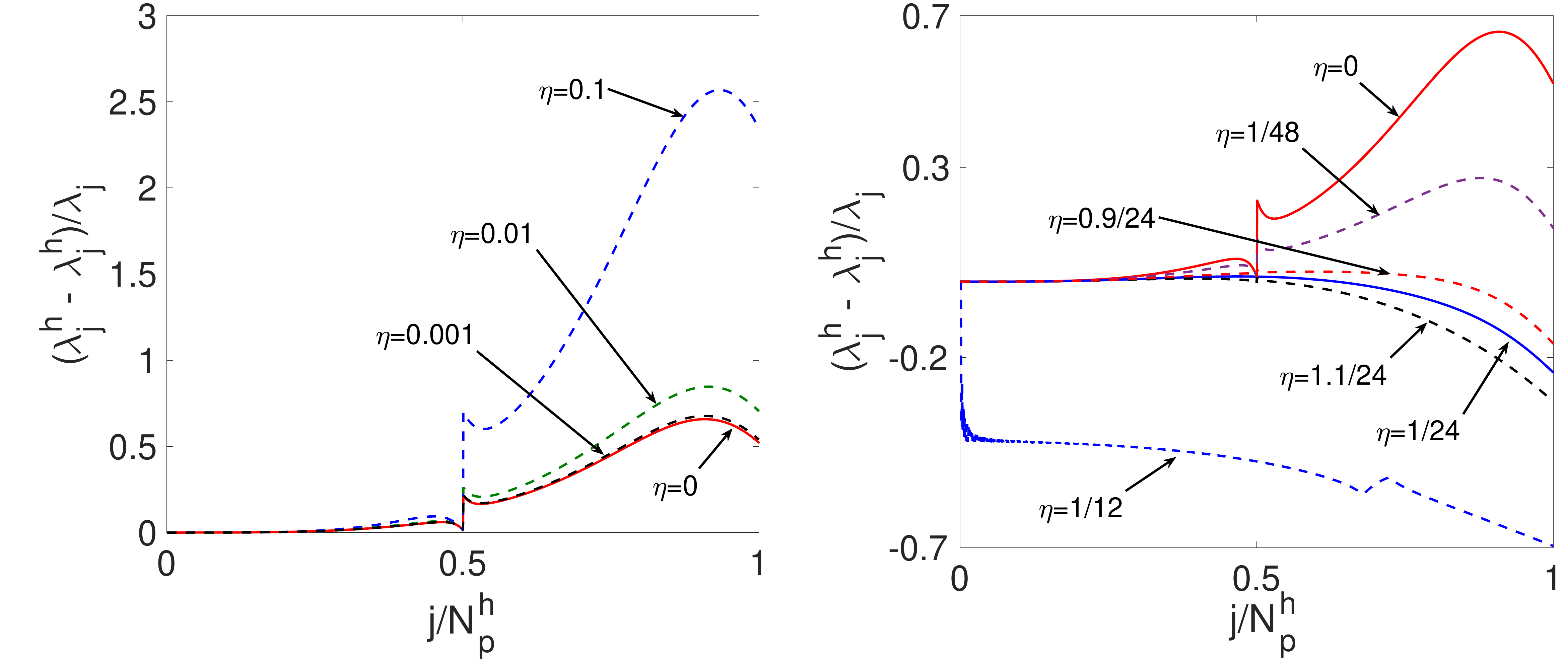} 
\vspace{-0.2cm}
\caption{Quadratic softFEM spectra in 1D with $N^h=1000$ elements using various softness parameters $\eta$. Left: $\hat a = a + \eta s$; Right: $\hat a = a - \eta s$.}
\label{fig:femp2n1000}
\end{figure}

To motivate the choice of the softness parameter $\eta=\frac{1}{2(p+1)(p+2)}=\frac{1}{24}$ for $p=2$, we show in Figure \ref{fig:femp2n1000} the softFEM discrete spectra using various values for the softness parameter $\eta$. In this experiment, we increase the mesh resolution to $N^h=1000$ elements. In the left panel of Figure \ref{fig:femp2n1000}, for the sake of illustration, we actually increase the stiffness, i.e., we set $\hat a := a + \eta s$. As expected, increasing $\eta$ merely worsens the results. Instead, in the right panel of Figure \ref{fig:femp2n1000}, we return to softFEM and consider $\hat a := a - \eta s$. We observe that the choice $\eta=\frac{1}{24}$ appears to deliver the best overall result concerning the accuracy of the discrete eigenvalues over the whole spectrum{. In} the high-frequency region, the accuracy of the discrete eigenvalues is sensitive to the value of the softness parameter. For 
reference, we also display the results for $\eta=\eta_{\max}=\frac{1}{2p(p+1)}=\frac{1}{12}$ which show that the limit value on the softness parameter derived in Theorem~\ref{thm:coe} is indeed sharp.

\begin{figure}[h!]
\centering
\includegraphics[height=5cm]{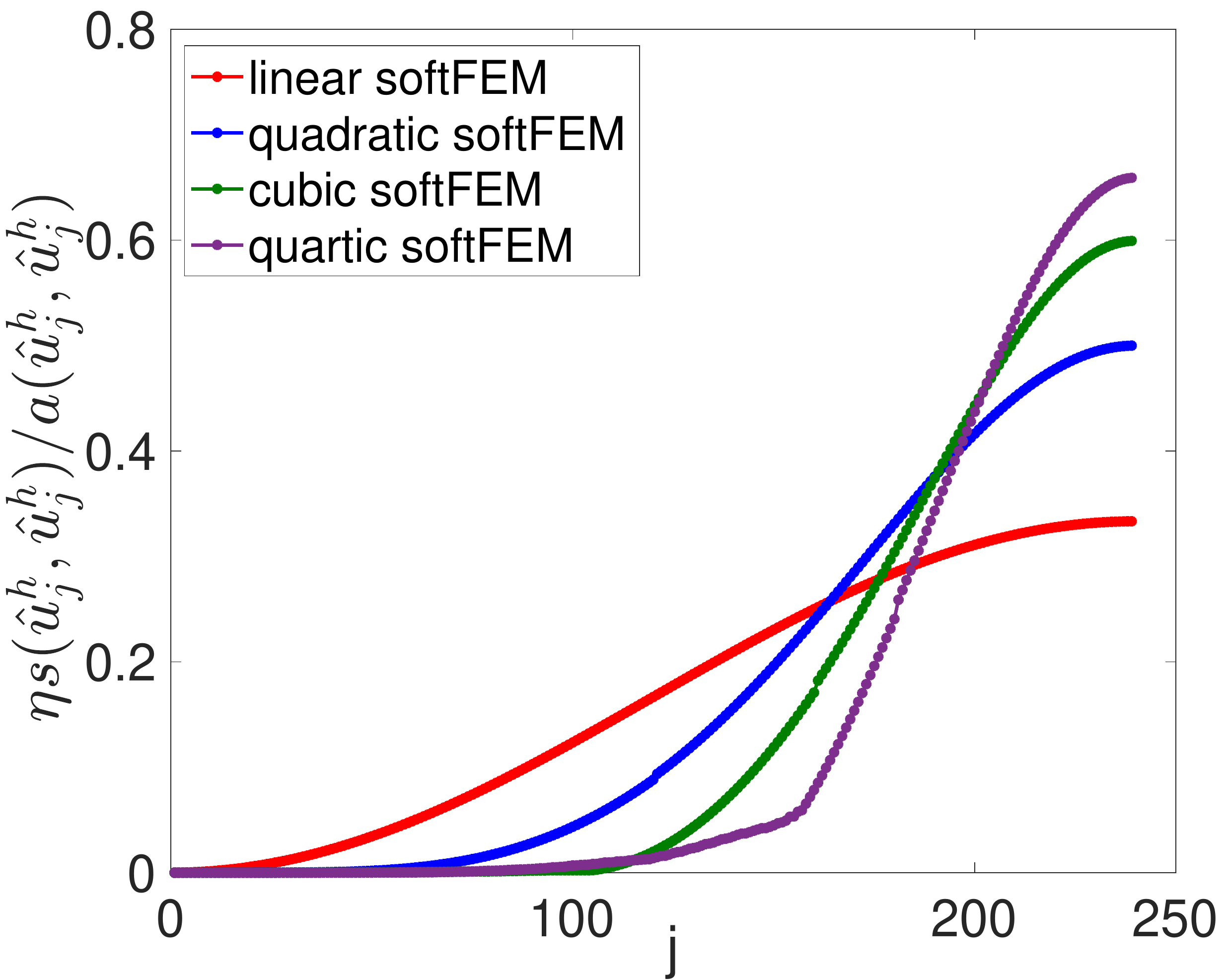}
\caption{Ratio $\eta s(\hat u^h_j, \hat u^h_j) / a(\hat u^h_j, \hat u^h_j)$ for softFEM eigenfunctions. The mesh is composed of 240, 120, 80, 60 uniform elements for $p\in\{1,\ldots, 4\}$, respectively.}
\label{fig:jau1d}
\end{figure}

In Figure \ref{fig:jau1d}{,} we present the ratio $\eta s(\hat u^h_j, \hat u^h_j) / a(\hat u^h_j, \hat u^h_j)$ for softFEM eigenfunctions. The mesh is composed of 240, 120, 80, 60 uniform elements for $p\in\{1,\ldots, 4\}$, respectively{, so that the number of eigenpairs is always the same}. As predicted by Theorem~\ref{thm:coe}, this ratio is always lower than one. We see that the amount of stiffness removed by softFEM is more substantial in the high-frequency region.

\begin{table}[ht]
\centering 
\begin{tabular}{| c | ccc | ccc c| cc |}
\hline
$p$ & $\lambda_{\min}^h$ &  $\lambda_{\max}^h$ & $ \hat \lambda_{\max}^h $ &  $\sigma$ &  $\hat \sigma$ & $\rho$ & $\varrho$ \\[0.1cm] \hline
1 & 9.8698 &   4.7991e5 &   3.1995e5 &   4.8624e4 &   3.2417e4 &  1.5000  & 33.33\% \\[0.1cm]
2 & 9.8696 &   2.3998e6 &   1.2000e6 &   2.4315e5 &   1.2158e5 &   1.9999 & 50.00\% \\[0.1cm]
3 & 9.8696 &   6.8046e6 &   2.7255e6 &   6.8945e5 &   2.7615e5 &   2.4967 & 59.95\% \\[0.1cm]
4 & 9.8696 &   1.5209e7 &   5.1587e6 &   1.5410e6 &   5.2269e5 &   2.9482 & 66.08\% \\[0.1cm]
5 & 9.8696 &   2.9555e7 &   9.1006e6 &   2.9946e6 &   9.2208e5 &   3.2476 & 69.21\% \\[0.1cm] \hline
 \end{tabular}
\caption{Minimal and maximal eigenvalues, condition numbers, stiffness reduction ratios, and percentages when using Galerkin FEM and softFEM for a mesh composed of $N^h=200$ uniform elements and polynomial degrees $p\in\{1,\ldots,5\}$.}
\label{tab:cond1d} 
\end{table}

Table \ref{tab:cond1d} shows the minimal and maximal eigenvalues, the condition numbers, the stiffness reduction ratios, and the percentages for Galerkin FEM and softFEM for a mesh composed of $N^h=200$ uniform elements and polynomial degrees $p\in\{1,\ldots,5\}$. (Recall that $\hat \lambda_{\min}^h \approx \lambda_{\min}^h$ so that we only show $\lambda_{\min}^h$ in the table.)
We observe that the stiffness reduction ratio increases with the polynomial degree, starting at $\rho=1.5$ for $p=1$ up to $\rho=3.2476$ for $p=5$. Thus, the benefit of using softFEM in tempering the condition number of the stiffness matrix becomes more pronounced as $p$ is increased. We also notice that the computed value for the stiffness reduction ratio $\rho$ is quite close to the optimal value $1+\frac{p}{2}$ resulting from Theorem~\ref{thm:bds} (see the lower bound in \eqref{eq:bds}). 

\subsection{Discrete spectrum for Galerkin FEM} \label{sec:sb}

The goal of this section is to briefly outline some basic facts about the spectrum
of Galerkin FEM for the 1D Laplace eigenvalue problem. 
We explore the polynomial degrees $p\in\{1,2,3\}$. For $p=1$, all the degrees of
freedom (dofs) in $V_p^h$ are attached to the $N^h_p$ mesh vertices. 
Letting $\Lambda:=\lambda h^2$, solving the GMEVP leads us to
look for nonzero vectors in the kernel of the following matrix of order 
$(N^h-1)\times(N^h-1)$:
\begin{equation} \label{eq:matrix_vv_1}
\mathbf{A}_{vv} := \Lambda \begin{bmatrix}
4 & 1 &  \\
1 & 4 & 1 &  \\
   & \ddots & \ddots & \ddots &  \\
  &  & 1 & 4 & 1  \\
&& & 1 & 4  \\
\end{bmatrix}
- 6\begin{bmatrix}
2 & -1 &  \\
-1 & 2 & -1 &  \\
   & \ddots & \ddots & \ddots &  \\
  &  & -1 & 2 & -1  \\
&& & -1 & 2  \\
\end{bmatrix}.
\end{equation}
For $p=2$, there are $N_2^h=2N^h-1$ dofs{. It} is interesting to order
first the $N^h-1$ dofs associated with the mesh vertices and then the $N^h$ dofs 
associated with the mesh elements{. The basis functions associated with these $N^h$ dofs} are bubble 
functions supported in a single mesh element. Solving the GMEVP problem leads us to look
for nonzero vectors in the kernel of the following matrix whose block decomposition reflects
the above partition into vertex and bubble dofs:
\begin{equation} \label{eq:block}
\begin{bmatrix}
\mathbf{A}_{vv} & \mathbf{0} \\
\mathbf{A}_{bv} & \mathbf{A}_{bb}
\end{bmatrix}.
\end{equation}
It turns out that there is one vector in the kernel of $\mathbf{A}_{bb}$ whose bubble dofs oscillate from one cell to the next one, and the corresponding eigenvalue is $\lambda_b=10h^{-2}$. The other vectors are obtained by considering the kernel of the block $\mathbf{A}_{vv}$  which admits the following structure:
\begin{equation} \label{eq:matrix_vv_2}
\begin{aligned}
\mathbf{A}_{vv} := &\Lambda^2 \begin{bmatrix}
6 & -1 &  \\
-1 & 6 & -1 &  \\
   & \ddots & \ddots & \ddots &  \\
  &  & -1 & 6 & -1  \\
&& & -1 & 6  \\
\end{bmatrix}
- 16 \Lambda \begin{bmatrix}
13 & 1 &  \\
1 & &13 & 1 &  \\
   & \ddots & \ddots & \ddots &  \\
  &  & 1 & 13 & 1  \\
&& & 1 & 13  \\
\end{bmatrix} \\
&
+  240 \begin{bmatrix}
2 & -1 &  \\
-1 & 2 & -1 &  \\
   & \ddots & \ddots & \ddots &  \\
  &  & -1 & 2 & -1  \\
&& & -1 & 2  \\
\end{bmatrix}.
\end{aligned}
\end{equation}
Finally, for $p=3$, there are $N_3^h=3N^h-1$ dofs{. We order}
first the $N^h-1$ dofs associated with the mesh vertices and then the $2N^h$ dofs 
associated with the mesh elements{. The basis functions associated with these $2N^h$ dofs} are bubble 
functions supported in a single mesh element (2 per element). 
Solving the GMEVP problem leads us to look for nonzero vectors in the kernel
of a matrix with the same block-structure as in~\eqref{eq:block}, but this time 
the block $\mathbf{A}_{bb}$ is two times larger. The kernel of $\mathbf{A}_{bb}$
is two-dimensional and the corresponding eigenfunctions are thus composed 
only of bubble functions. Moreover, we have
\begin{equation} \label{eq:matrix_vv_3}
\begin{aligned}
\mathbf{A}_{vv} := & 
\Lambda^3 
\begin{bmatrix}
8 & 1 &  \\
1 & 8 & 1 &  \\
   & \ddots & \ddots & \ddots &  \\
  &  & 1 & 8 & 1  \\
&& & 1 & 8  \\
\end{bmatrix}
- 30 \Lambda^2
\begin{bmatrix}
36 & -1 &  \\
-1 & 36 & -1 &  \\
   & \ddots & \ddots & \ddots &  \\
  &  & -1 & 36 & -1  \\
&& & -1 & 36  \\
\end{bmatrix} \\
&\quad  + 360\Lambda
\begin{bmatrix}
64 & 3 &  \\
3 & 64 & 3 &  \\
   & \ddots & \ddots & \ddots &  \\
  &  & 3 & 64 & 3  \\
&& & 3 & 64  \\
\end{bmatrix}
- 25200
\begin{bmatrix}
2 & -1 &  \\
-1 & 2 & -1 &  \\
   & \ddots & \ddots & \ddots &  \\
  &  & -1 & 2 & -1  \\
&& & -1 & 2  \\
\end{bmatrix}.
\end{aligned}
\end{equation}
For all $p\in\{1,2,3\}$, one can readily verify that the matrix $\mathbf{A}_{vv}$ has a non-trivial kernel
if and only if $\Lambda$ is a root of the following polynomials (the subscript refers to
the polynomial degree):
\begin{equation} \label{eq:p1p3cp}
\begin{aligned}
f_1(\Lambda) & = (2 + \zeta_j ) \Lambda - 6 (1 - \zeta_j ), \\
f_2(\Lambda) & = 2( 3 - \zeta_j ) \Lambda^2 - 16 (13 + 2 \zeta_j ) \Lambda + 480 (1 - \zeta _j), \\
f_3(\Lambda) & = ( 4 + \zeta_j ) \Lambda^3 - 30( 18 - \zeta_j ) \Lambda^2 + 360 (32 + 3 \zeta_j ) \Lambda - 25200 (1 - \zeta_j ),
\end{aligned}
\end{equation} 
where $\zeta_j := \cos(\pi t_j)$, $t_j:=jh$ and $j\in\{1,\ldots,N^h-1\}$. By replacing $\zeta_j$ by the continuous variable $\zeta:=\cos(\pi t)$ with $t\in (0,1)$, one obtains one branch of eigenvalues for $p=1$, two branches of eigenvalues for $p=2$, and three branches of eigenvalues for $p=3$. Each branch contains $N^h-1$ eigenvalues. For $p\in\{2,3\}$, the spectrum is completed by the one or two eigenvalues associated with the eigenfunction(s) composed of bubble functions only. 

\begin{figure}[h]
\centering
\includegraphics[height=4.8cm]{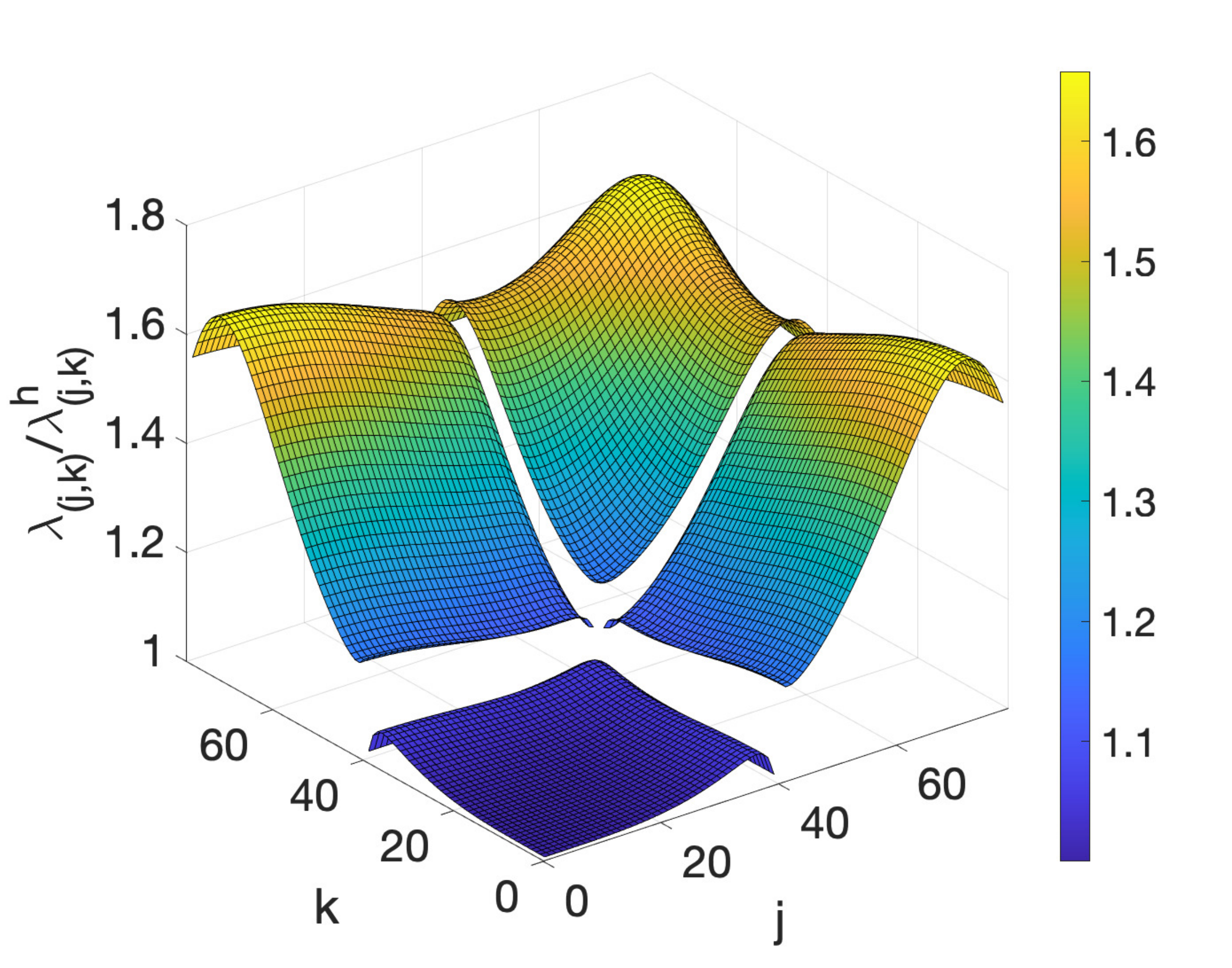} 
\includegraphics[height=4.8cm]{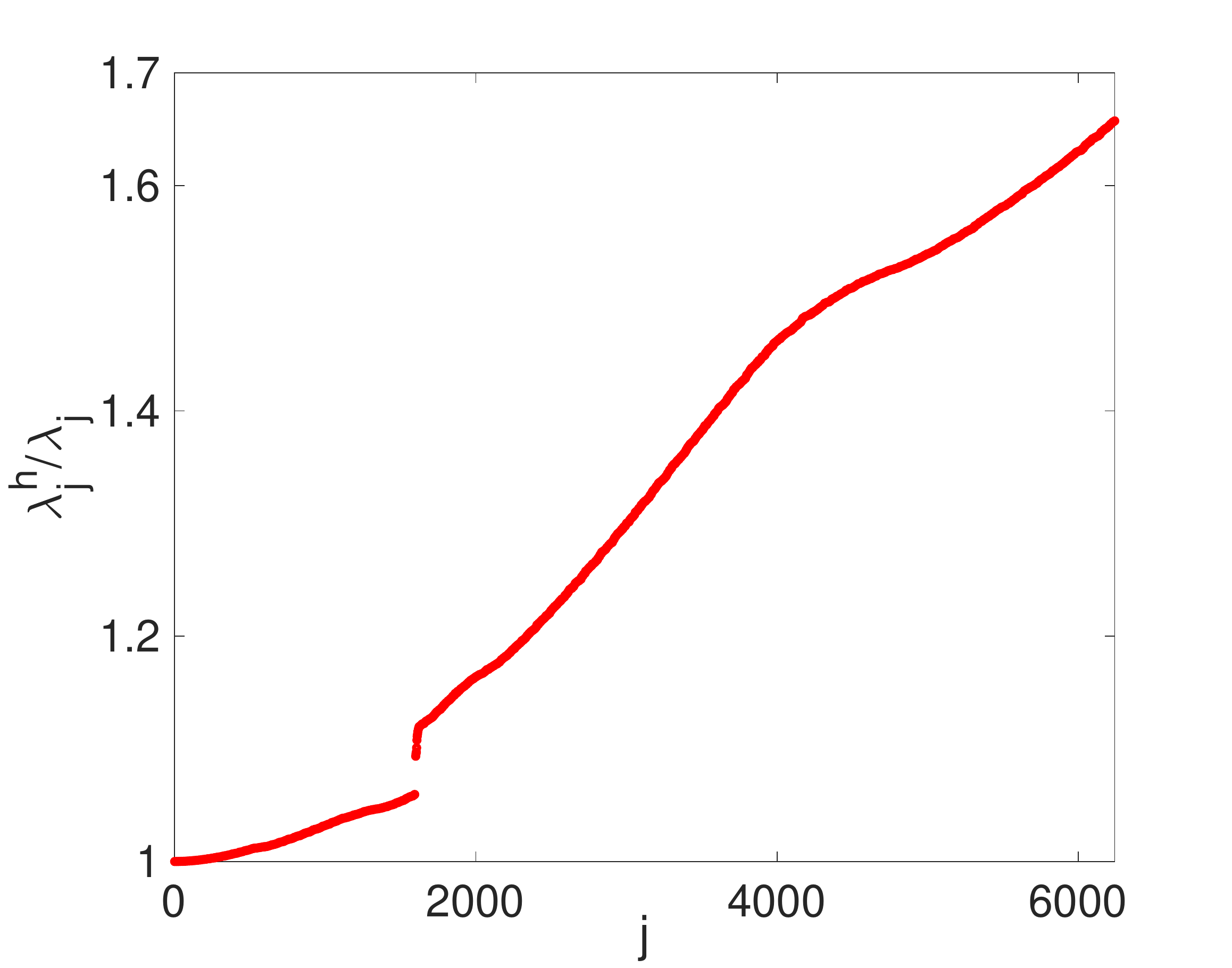} 
\caption{Quadratic FEM approximate spectrum with $N^h=40\times40$ elements for the 2D Laplace eigenvalue problem. Left: eigenvalues sorted in each dimension. Right: eigenvalues sorted in 2D.}
\label{fig:fem2dp2n40}
\end{figure}

In the literature, one refers to these latter eigenvalues as stopping band(s), whereas the branch associated with the lowest eigenvalues is called acoustical branch and the other branches are called optical branches.  
For instance, \cite{brillouin} reported that quadratic finite elements for the 1D Laplace eigenvalue problem delivered an acoustical branch (low-frequency region) and an optical branch (high-frequency region) separated by one stopping band. We refer the reader to the left plots in Figure~\ref{fig:fem1D} for an illustration of these notions. We also observe that the notions of acoustical and optical branches as well as stopping bands depend on the sorting of the eigenvalues and that some overlap between the branches can happen in multiple dimensions; see Figure~\ref{fig:fem2dp2n40} for an illustration in 2D.

\section{SoftFEM on more challenging numerical examples} \label{sec:num}

In this section, we present more challenging numerical tests to illustrate the performances of softFEM. We consider Laplace eigenvalue problems on tensor-product meshes in Section~\ref{sec:num_lap_tensor}, 
elliptic eigenvalue problems and non-uniform meshes in 1D in Section~\ref{sec:num_ell_nonuni}, and finally simplicial meshes and L-shaped domains in Section~\ref{sec:unmsh}. 
The exact eigenpairs of the Laplace eigenvalue problems are known {for the problems in Section~\ref{sec:num_lap_tensor}}, whereas for the problems in Sections~\ref{sec:num_ell_nonuni} and \ref{sec:unmsh}, we use a higher-order method with a large number of elements to produce reference eigenpairs so as to quantify the approximation errors. 

\subsection{Laplace eigenvalue problems on tensor-product meshes}\label{sec:num_lap_tensor}

We consider the spectral problem \eqref{eq:pde} posed on $\Omega=(0,1)^d$, $d\in\{2,3\}$, with $\kappa=1$. For $d=2$, the exact eigenvalues and eigenfunctions are respectively for all $i,j = 1, 2, \ldots$,
\begin{equation*}
\lambda_{ij} = ( i^2 + j^2 ) \pi^2, \qquad u_{ij}(x,y) = c_{ij} \sin( i\pi x)\sin( j\pi y),
\end{equation*}
for some normalization constant $c_{ij}>0$, whereas for $d=3$, the exact eigenvalues and eigenfunctions are respectively for all $k,l,m = 1, 2, \ldots$,
\begin{equation*}
{\lambda _{klm}} = {(k^2+l^2+m^2)}{\pi ^2},\quad u_{klm}(x,y,z) =  c_{klm}\sin (k\pi x) \sin (l\pi y) \sin (m\pi z),
\end{equation*}
for some normalization constant $c_{klm}>0$. For the Galerkin FEM and softFEM approximation,
we use uniform tensor-product meshes. Theorem~\ref{thm:coe} shows that admissible values for the softness parameter are $\eta\in [0,\eta_{\max})$ with $\eta_{\max}=\frac{1}{2p(p+1)}$. Motivated by the 1D numerical experiments reported Section~\ref{sec:1d}, we take again $\eta=\frac{1}{2(p+1)(p+2)}$.

\begin{figure}[h!]
\centering
\includegraphics[height=5.5cm]{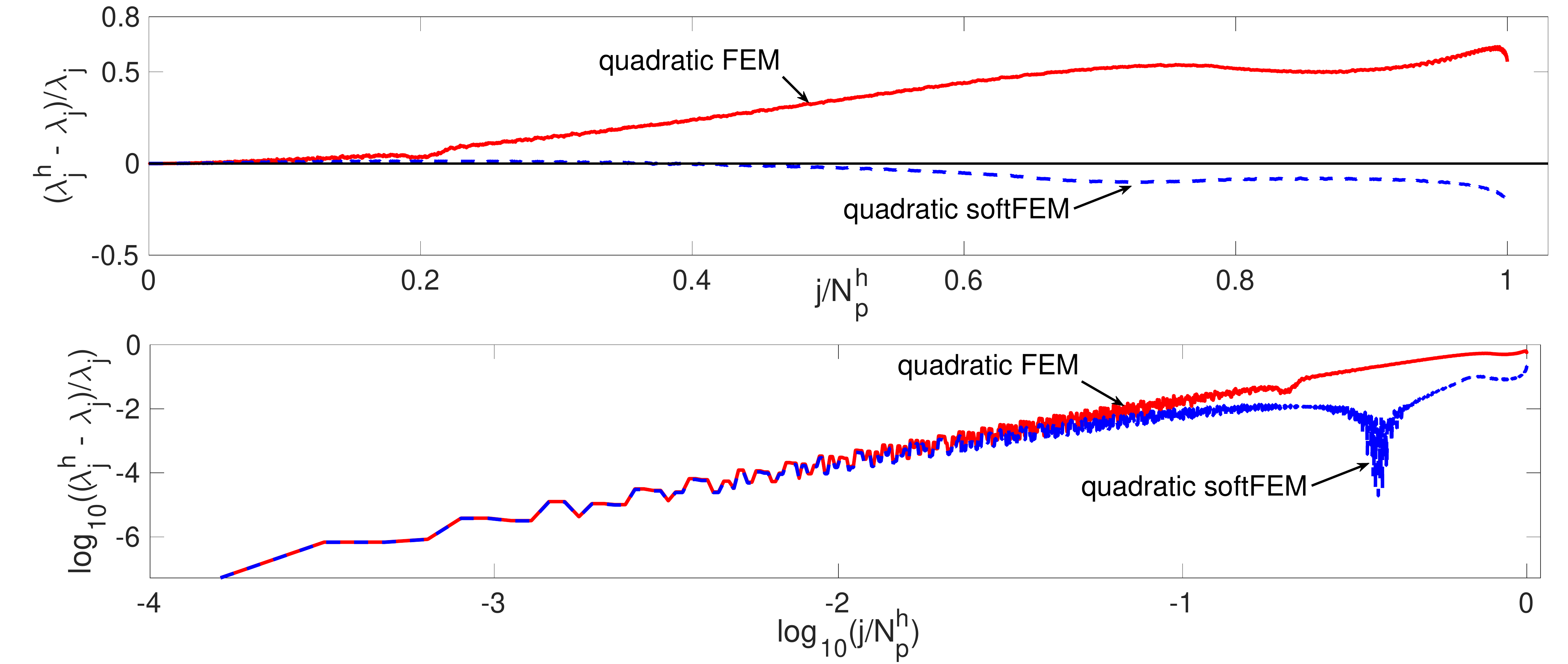}
\caption{Relative eigenvalue errors for the 2D Laplace eigenvalue problem when using quadratic Galerkin FEM and softFEM with $40^2$ elements. }
\label{fig:fem2dp2}
\end{figure}

\begin{figure}[h!]
\centering
\includegraphics[height=5.5cm]{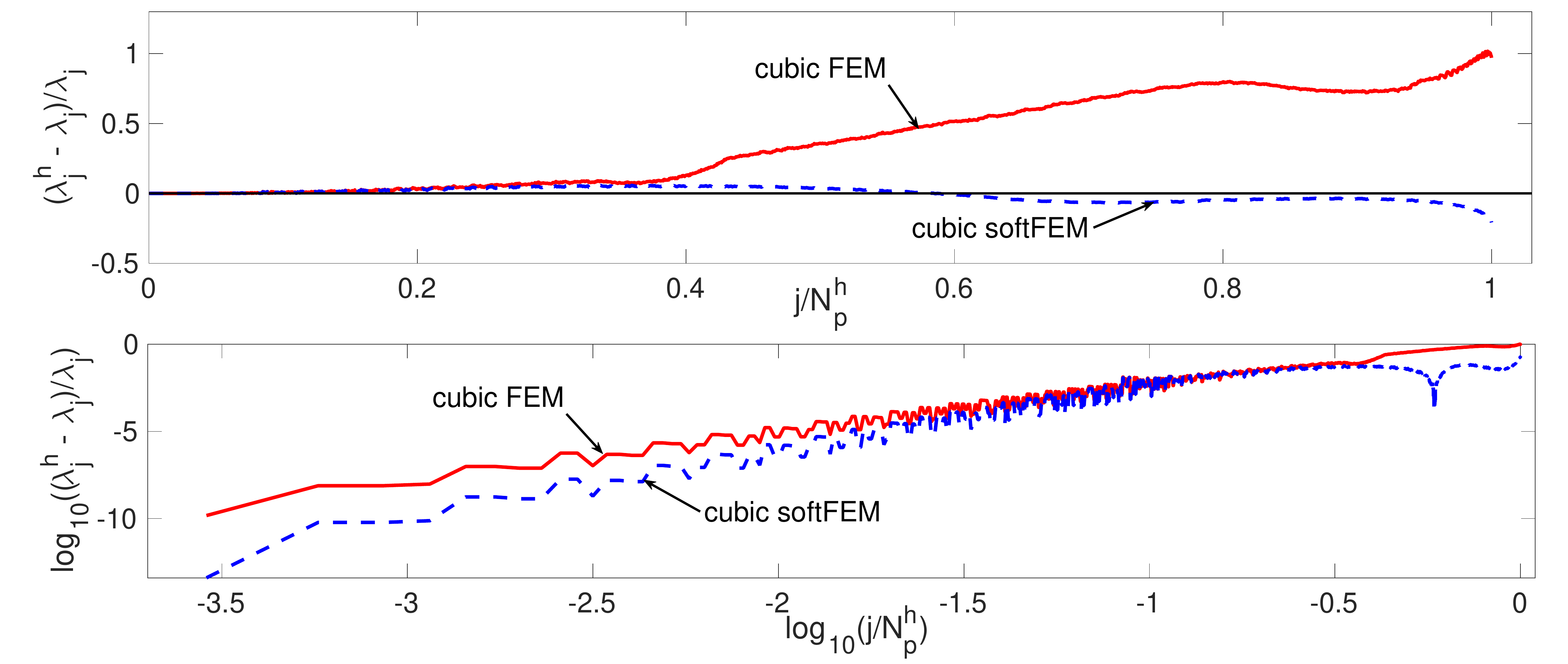}
\caption{Relative eigenvalue errors for the 2D Laplace eigenvalue problem when using cubic Galerkin FEM and softFEM with $20^2$ elements. }
\label{fig:fem2dp3}
\end{figure}

Figures~\ref{fig:fem2dp2} and~\ref{fig:fem2dp3} show the relative eigenvalue errors when using quadratic and cubic Galerkin FEM and softFEM in 2D. For quadratic elements, we use a uniform mesh with $40\times40$ elements, whereas for cubic elements, we use a uniform mesh with $20\times20$ elements. Figure \ref{fig:fem3d} shows the {relative} eigenvalue errors for the 3D problem with $20\times20\times 20$ elements and $p\in\{2,3,4\}$. We observe in these plots that softFEM significantly improves the accuracy in the high-frequency region. Moreover, the plots using the log-log scale indicate that the spectral accuracy is maintained for quadratic elements and even improved for cubic elements in the low-frequency region. The convergence rates for the errors are optimal, and we omit them for brevity. 

\begin{figure}[h!]
\centering
\includegraphics[height=7.5cm]{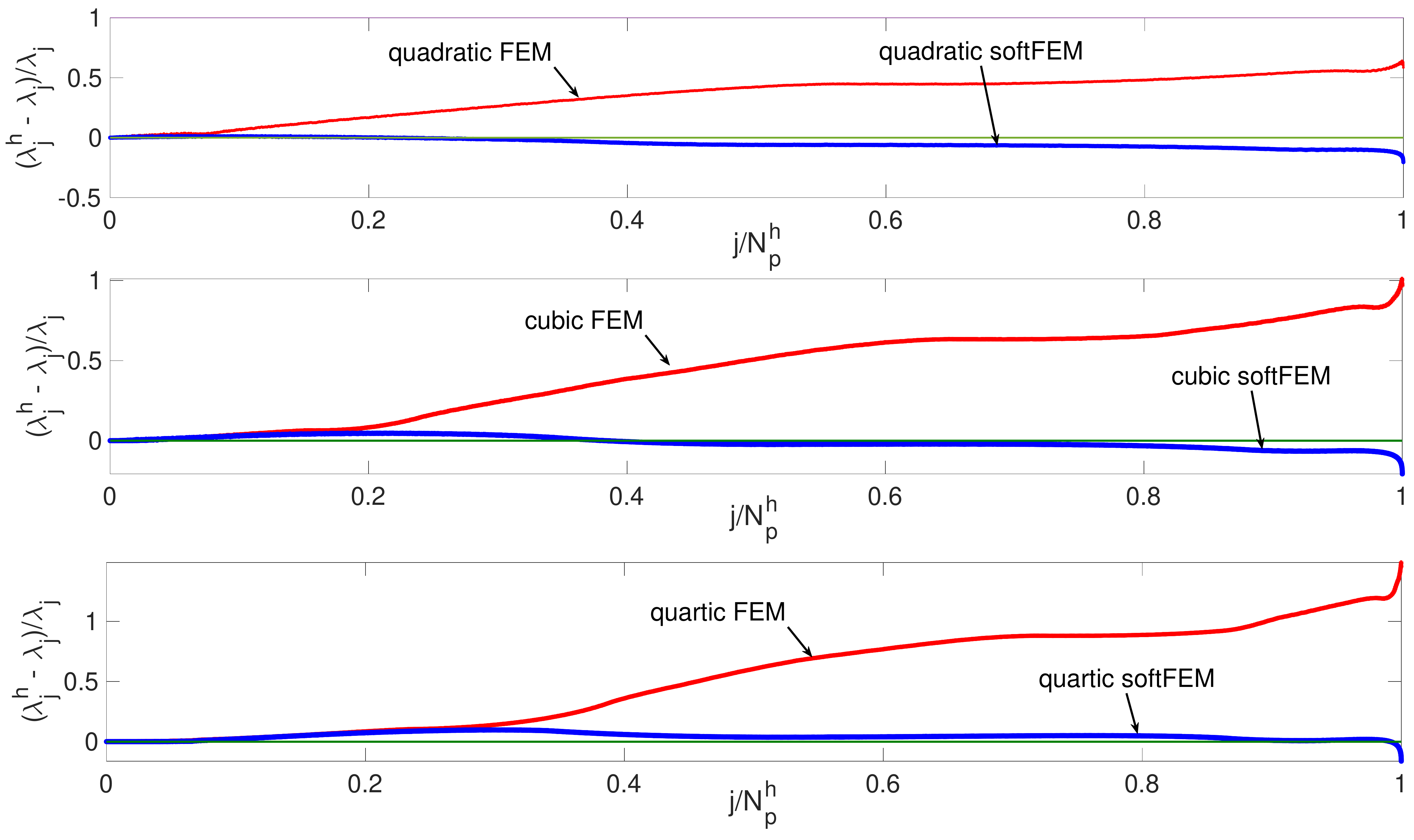}
\vspace{-0.4cm}
\caption{Relative eigenvalue errors for the 3D Laplace eigenvalue problem when using FEM and softFEM with $p=2,3,4$. }
\label{fig:fem3d}
\end{figure}

\begin{figure}[h!]
\centering
\includegraphics[height=5cm]{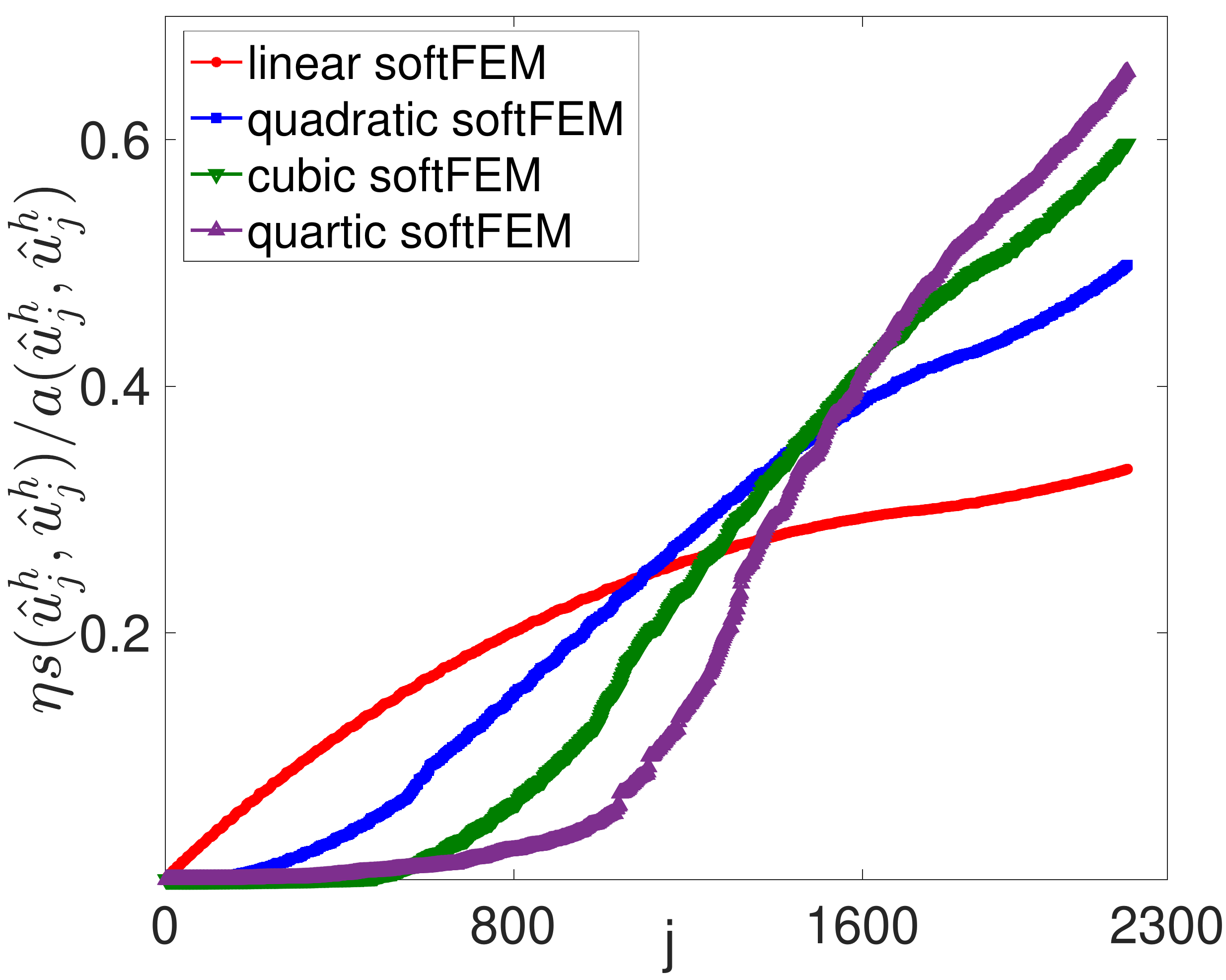}
\includegraphics[height=5cm]{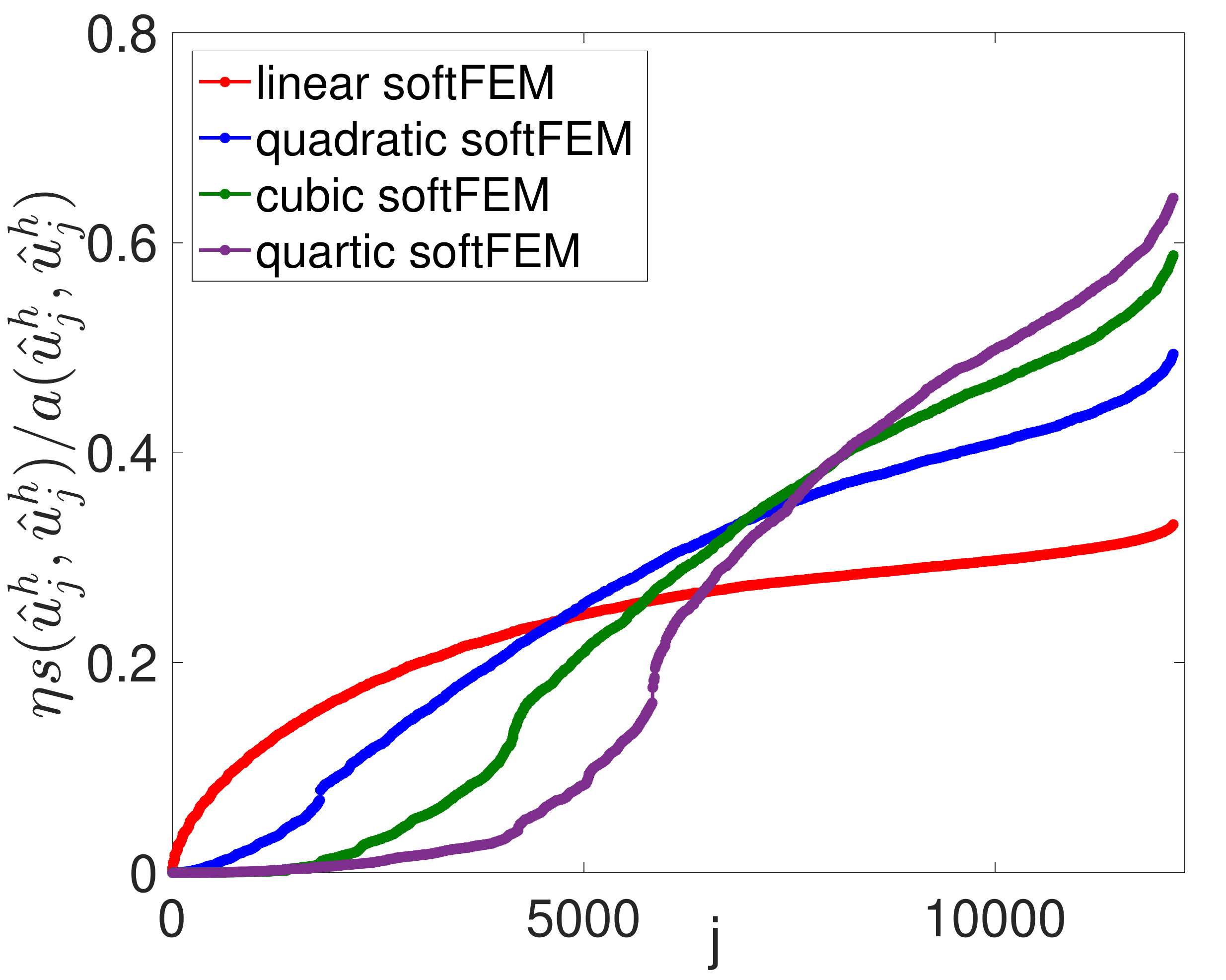}
\vspace{-0.2cm}
\caption{Ratio $\eta s(\hat u^h_j, \hat u^h_j) / a(\hat u^h_j, \hat u^h_j)$ for the softFEM eigenfunctions for the Laplace eigenvalue problem in 2D (left) and 3D (right). }
\label{fig:jau}
\end{figure}

Figure \ref{fig:jau} shows the ratio $\eta s(\hat u^h_j, \hat u^h_j) / a(\hat u^h_j, \hat u^h_j)$ for the softFEM eigenfunctions in both the 2D and 3D settings. In 2D, there are $48\times 48$, $24\times 24$, $16\times 16$, and $12\times12$ uniform elements for $p\in\{1,\ldots, 4\}$, respectively, whereas in 3D, there are $24\times 24\times24$, $12\times 12\times12$, $8\times 8\times 8$, and $6\times 6\times6$ uniform elements for $p\in\{1,\ldots, 4\}$, respectively. These results {essentially show} how much stiffness is removed from the eigenfunctions by means of softFEM. The fact that the ratio $\eta s(\hat u^h_j, \hat u^h_j) / a(\hat u^h_j, \hat u^h_j)$ is more pronounced in the high-frequency region corroborates the reduction of the spectral errors in this region. Finally, we mention that the stiffness reduction ratios and percentages are quite close to those reported in 1D, that is, $\rho \approx 1 + \frac{p}{2}$ and $\varrho \approx 100\frac{p}{p+2}\% $ for $p\in\{1,\ldots, 4\}$ in both 2D and 3D. 

\subsection{Elliptic eigenvalue problems and non-uniform meshes in 1D} \label{sec:num_ell_nonuni}

\begin{figure}[h!]
\centering
\includegraphics[height=8cm]{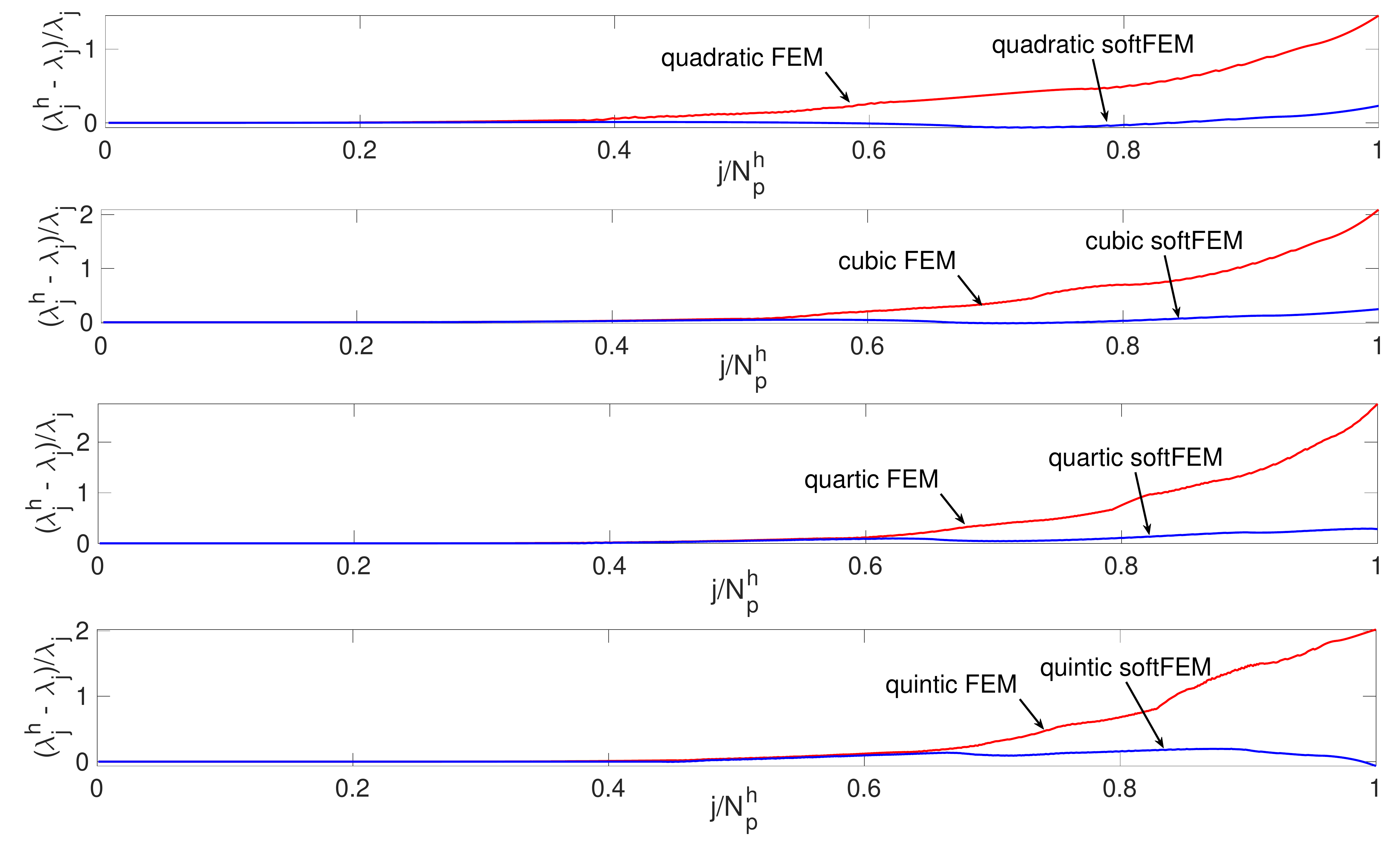}
\caption{Relative eigenvalue errors for the elliptic eigenvalue problem \eqref{eq:pde} in 1D with $\kappa(x) := e^{x \sin(2\pi x)}$ when using Galerkin FEM and softFEM with $p\in\{2,3,4,5\}$. The mesh has $N^h=200$ uniform elements.}
\label{fig:fem1dgevp}
\end{figure}

We now consider the 1D elliptic eigenvalue problem \eqref{eq:pde} with $\kappa(x) := e^{x \sin(2\pi x)}$, so that $\kappa_{\max} \approx 1.34$ and $\kappa_{\min} \approx 0.46$. The exact eigenpairs are approximated using Galerkin FEM with $C^6$ septic B-spline basis functions and a mesh composed of $N^h=1000$ elements. Figure \ref{fig:fem1dgevp} compares the relative eigenvalue errors for Galerkin FEM and softFEM on a uniform mesh composed of $N^h=200$ elements and polynomial degrees $p\in\{2,3,4,5\}$. We observe that softFEM reduces the spectrum errors, especially in the high-frequency region. The convergence rates for the errors are optimal, and we omit them here for brevity.

Table \ref{tab:cond1dgevp} shows the smallest and largest eigenvalues, the condition numbers, the stiffness reduction ratios, and the percentages for Galerkin FEM and softFEM. 
In all cases, we observe that softFEM leads to smaller largest eigenvalues and hence to smaller condition numbers. The stiffness reduction ratio is about $\rho = \sigma / \hat \sigma \approx 1+\frac{p}{2}$ while the percentage is about $\varrho\approx 100\frac{p}{p+2}\%$; this is consistent with the 1D results reported in Section~\ref{sec:num_1D}.

\begin{table}[ht]
\centering 
\begin{tabular}{| c | ccc | cccc | cc |}
\hline
$p$ & $\lambda_{\min}^h$ &  $\lambda_{\max}^h$ & $ \hat \lambda_{\max}^h $ &  $\sigma$ &  $\hat \sigma$ & $\rho$ & $\varrho$ \\[0.1cm] \hline
1 & 8.2832 &   6.3326e5 &   4.2263e5 &   7.6451e4 &   5.1023e4 &   1.4984 &  33.26\% \\[0.1cm]
2 & 8.2829 &   3.1795e6 &   1.5936e6 &   3.8386e5 &   1.9240e5 &   1.9951 & 49.88\% \\[0.1cm]
3 & 8.2829 &   9.0280e6 &   3.6298e6 &   1.0900e6 &   4.3823e5  &  2.4872 & 59.79\% \\[0.1cm]
4 & 8.2829 &   2.0194e7 &   6.8865e6 &   2.4380e6 &   8.3141e5 &   2.9323 & 65.90\% \\[0.1cm]
5 & 8.2829 &   3.9263e7 &   1.2129e7 &   4.7402e6 &   1.4643e6 &   3.2371 & 69.11\% \\[0.1cm] \hline
 \end{tabular}
\caption{Minimal and maximal eigenvalues, condition numbers, stiffness reduction ratios, and percentages for the 1D elliptic eigenvalue problem with $\kappa(x) := e^{x \sin(2\pi x)}$ when using Galerkin FEM and softFEM on a uniform mesh composed of $N^h=200$ elements.}
\label{tab:cond1dgevp} 
\end{table}

Figure \ref{fig:fem1dnu} compares the relative eigenvalue errors for the 1D Laplace eigenvalue problem when using Galerkin FEM and softFEM with $p\in\{2,3,4,5\}$ on a non-uniform mesh composed of $N^h=10$ elements. The mesh nodes have been randomly set to $\{0, 0.1, 0.18, 0.29, 0.41, 0.5, 0.59, 0.66, 0.81, 0.92, 1\}$. Reference eigenvalues to evaluate the errors are computed as above. We observe that the improvement offered by softFEM over Galerkin FEM is similar to the one observed on uniform meshes. Table \ref{tab:cond1dnu} reports the smallest and largest eigenvalues, the condition numbers, the stiffness reduction ratios, and the percentages. We observe that the stiffness reduction ratios are slightly larger than when using uniform meshes (compare with Table~\ref{tab:cond1dgevp}). 

\begin{figure}[h!]
\centering
\includegraphics[height=8cm]{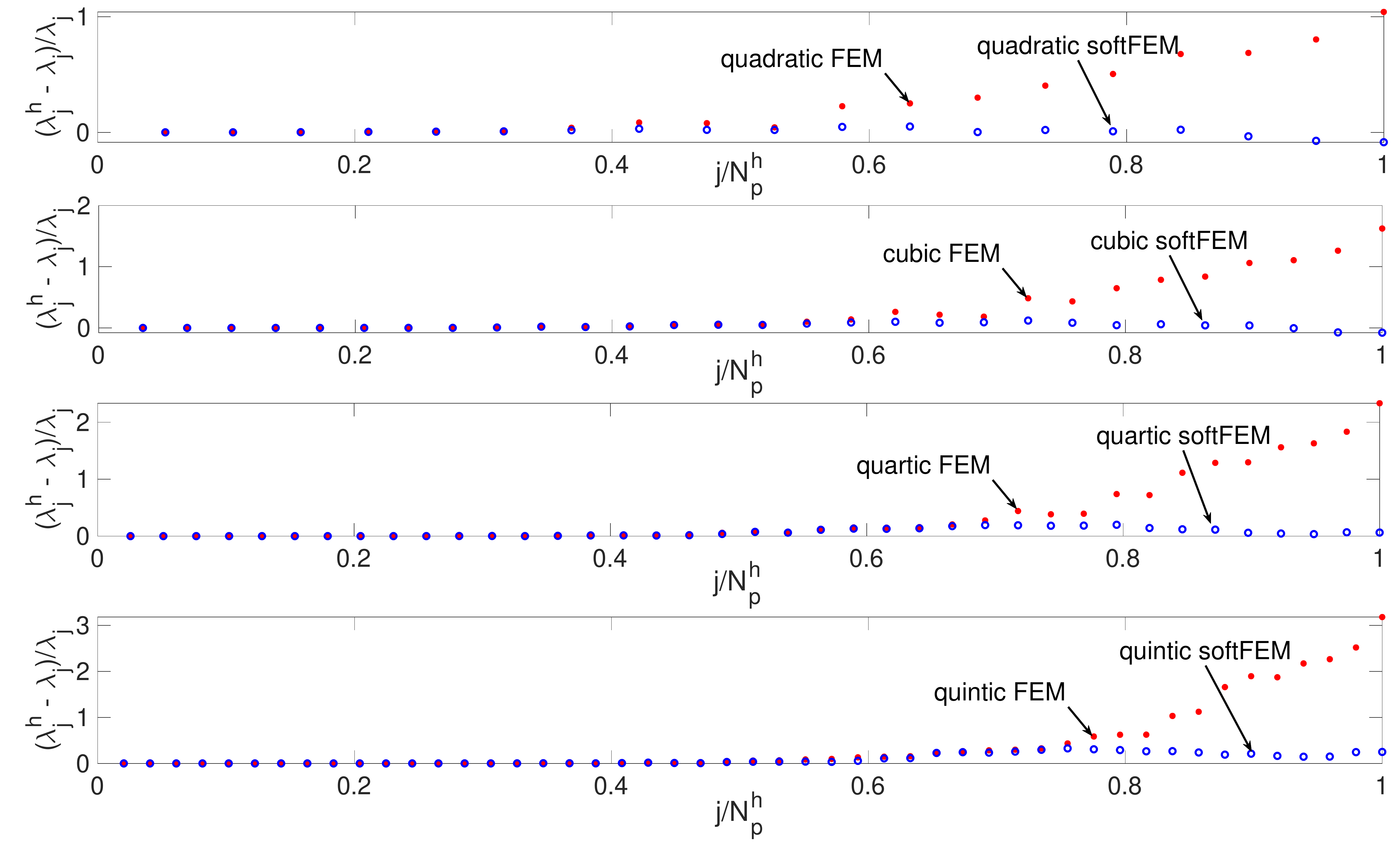}
\caption{Relative eigenvalue errors for the 1D Laplace eigenvalue problem when using Galerkin FEM and softFEM with $p\in\{2,3,4,5\}$ on a non-uniform mesh composed of $N^h=10$ elements. The mesh nodes have been randomly set to $\{0, 0.1, 0.18, 0.29, 0.41, 0.5, 0.59, 0.66, 0.81, 0.92, 1\}$.}
\label{fig:fem1dnu}
\end{figure}

\begin{table}[ht]
\centering 
\begin{tabular}{| c | ccc | cccc | cc |}
\hline
$p$ & $\lambda_{\min}^h$ &  $\lambda_{\max}^h$ & $ \hat \lambda_{\max}^h $ &  $\sigma$ &  $\hat \sigma$ & $\rho$ & $\varrho$ \\[0.1cm] \hline
1 & 9.9653 &    1.2631e3 &    8.0985e2 &    1.2675e2 &    8.1267e1 &    1.5597 &    35.88\% \\[0.1cm]
2 & 9.8698 &    7.2767e3 &    3.2585e3 &    7.3727e2 &    3.3014e2 &    2.2332 &    55.22\% \\[0.1cm]
3 & 9.8696 &    2.1782e4 &    7.6596e3 &    2.2070e3 &    7.7608e2 &    2.8438 &    64.84\% \\[0.1cm]
4 & 9.8696 &    5.0056e4 &    1.5948e4 &    5.0717e3 &    1.6159e3 &    3.1387 &    68.14\% \\[0.1cm]
5 & 9.8696 &    9.9119e4 &    2.9618e4 &    1.0043e4 &    3.0009e3 &    3.3466 &    70.12\% \\[0.1cm] \hline
 \end{tabular}
\caption{Minimal and maximal eigenvalues, condition numbers, stiffness reduction ratios, and percentages for the 1D Laplace eigenvalue problem when using Galerkin FEM and softFEM on a non-uniform mesh composed of $N^h=10$ elements. The mesh nodes have been randomly set to $\{0, 0.1, 0.18, 0.29, 0.41, 0.5, 0.59, 0.66, 0.81, 0.92, 1\}$.}
\label{tab:cond1dnu} 
\end{table}

\subsection{Simplicial meshes and L-shaped domain} \label{sec:unmsh}

In this section{,} we consider the 2D Laplace eigenvalue problem posed on the unit square domain or on the L-shaped domain, and we use simplicial meshes (triangulations) as depicted in Figure \ref{fig:2dmesh}.
Theorem~\ref{thm:coe} shows that admissible values for the softness parameter on simplicial meshes are $\eta\in [0,\eta_{\max})$ with $\eta_{\max}=\frac{1}{2p(p+d-1)}=\frac{1}{2p(p+1)}$ if $d=2$. Motivated by the 1D numerical experiments reported in the previous section, we take again $\eta=\frac{1}{2(p+1)(p+2)}$.

\begin{figure}[h!]
\centering
\includegraphics[height=5cm]{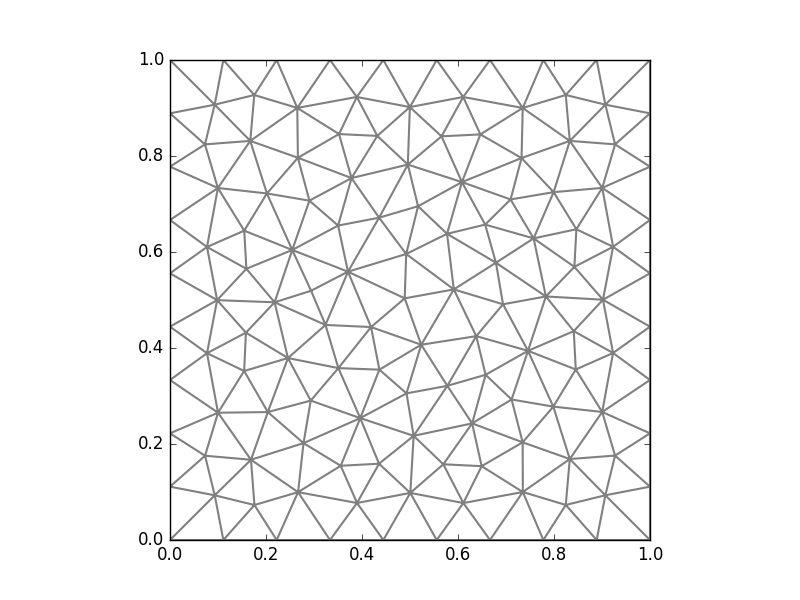} 
\hspace{-1cm}
\includegraphics[height=5cm]{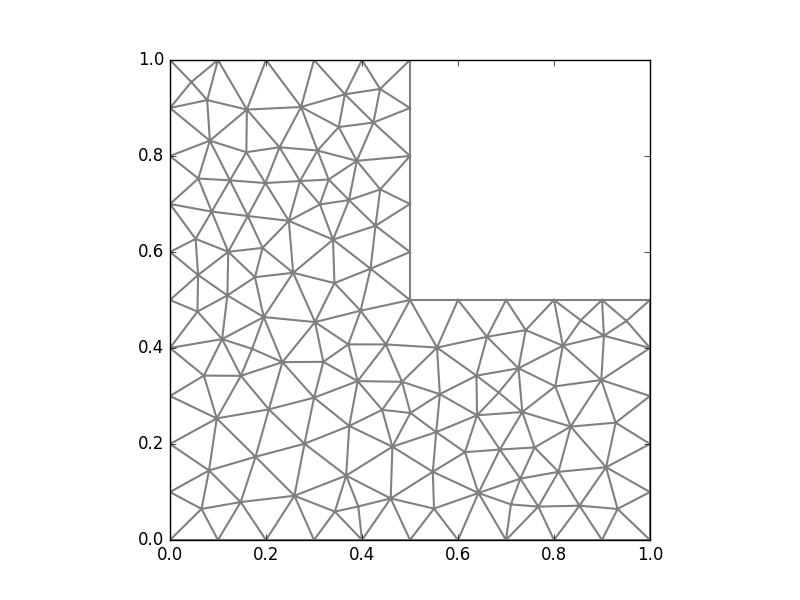}
\vspace{-0.4cm}
\caption{Unstructured meshes for the unit square domain (left) and the L-shaped domain (right).}
\label{fig:2dmesh}
\end{figure}

\begin{figure}[h!]
\centering
\includegraphics[height=7.5cm]{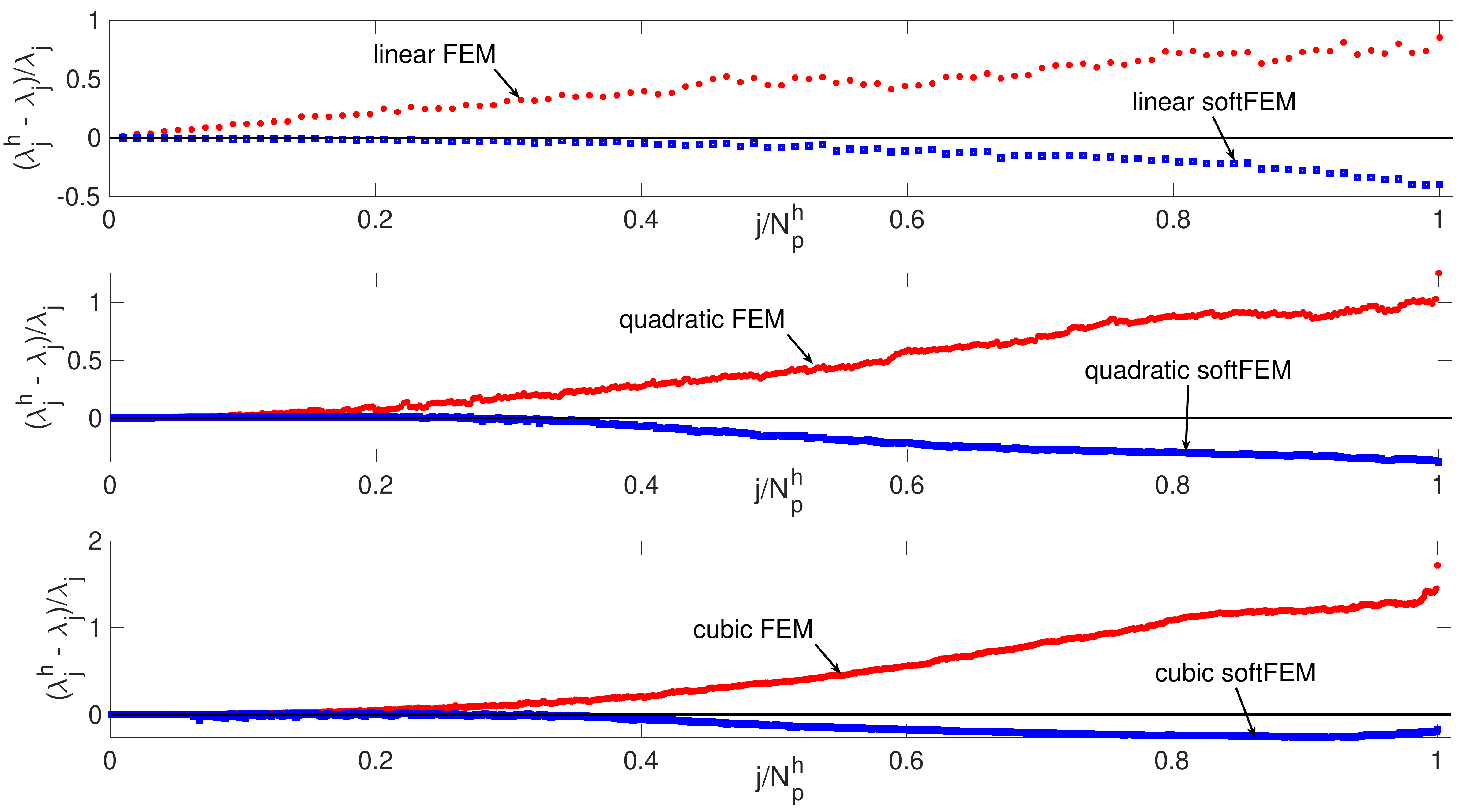}
\caption{Relative eigenvalue errors for the 2D Laplace eigenvalue problem on the unit square domain when using Galerkin FEM and softFEM with $p\in\{1,2,3\}$ and an unstructured mesh.}
\label{fig:fem2dum}
\end{figure}

\begin{figure}[h!]
\centering
\includegraphics[height=7.5cm]{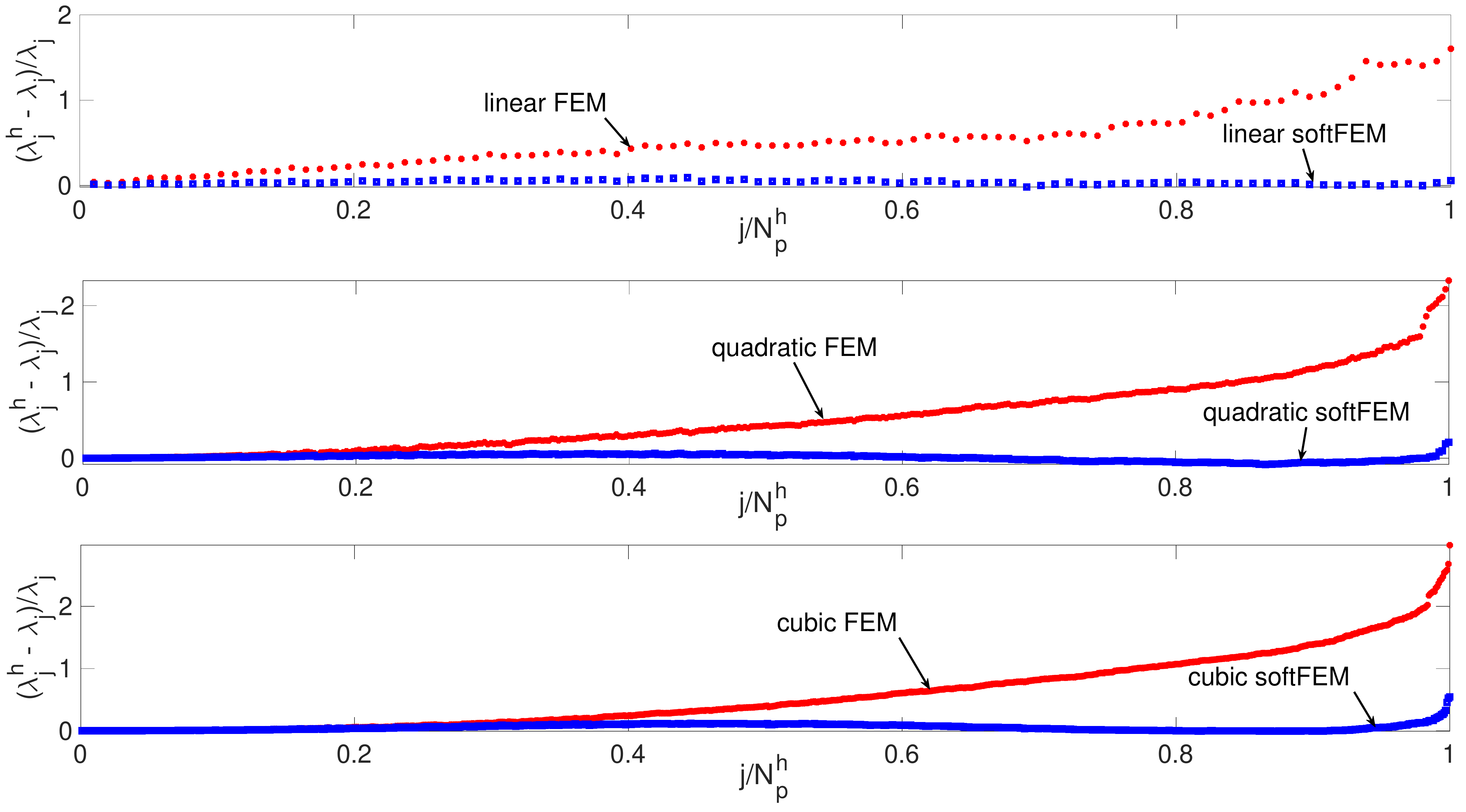}
\caption{Relative eigenvalue errors for the 2D Laplace eigenvalue problem on the L-shaped domain when using Galerkin FEM and softFEM with $p\in\{1,2,3\}$ and an unstructured mesh.}
\label{fig:fem2dlshape}
\end{figure}

Figures \ref{fig:fem2dum} and \ref{fig:fem2dlshape} compare the relative eigenvalue errors for the 2D Laplace eigenvalue problem on the unit square domain and the L-shaped domain, respectively, when using Galerkin FEM and softFEM with $p\in\{1,2,3\}$ and an unstructured mesh (triangulation). As observed in the previous numerical experiments, softFEM leads to smaller spectral errors than Galerkin FEM especially in the high-frequency region. 

\begin{table}[h!]
\centering 
\begin{tabular}{| c | ccc | ccc c | cc |}
\hline
$p$ & $\lambda_{\min}^h$ &  $\lambda_{\max}^h$ & $ \hat \lambda_{\max}^h $ &  $\sigma$ & $\hat \sigma$ & $\rho$ & $\varrho$ \\[0.1cm] \hline
%
1& 2.0020e1 &    2.9992e3 &    9.8013e2 &    1.4981e2 &    4.8957e1 &    3.0600 &    67.32\% \\[0.1cm]
2& 1.9740e1 &    1.5224e4 &    4.1819e3 &    7.7122e2 &    2.1185e2  &   3.6404 &    72.53\% \\[0.1cm]
3& 1.9739e1 &    4.0719e4 &    1.2356e4 &    2.0628e3 &    6.2598e2 &    3.2954 &    69.65\% \\[0.1cm] \hline
\end{tabular}
\caption{Minimal and maximal eigenvalues, condition numbers, stiffness reduction ratios, and percentages for the 2D Laplace eigenvalue problem on the unit square domain when using Galerkin FEM and softFEM with $p\in\{1,2,3\}$ and an unstructured mesh.}
\label{tab:fem2dum} 
\end{table}

\begin{table}[h!]
\centering 
\begin{tabular}{| c | ccc | cccc | cc |}
\hline
$p$ & $\lambda_{\min}^h$ &  $\lambda_{\max}^h$ & $ \hat \lambda_{\max}^h $ &  $\sigma$ &  $\hat \sigma$ & $\rho$ & $\varrho$ \\[0.1cm] \hline
%
1& 4.0162e1 &    4.7287e3 &    1.9228e3 &    1.1774e2  &   4.7875e1 &    2.4593 &    59.34\% \\[0.1cm]
2& 3.8707e1 &    2.5394e4 &    9.2240e3 &    6.5605e2 &    2.3830e2 &    2.7530 &    63.68\% \\[0.1cm]
3& 3.8619e1 &    7.1172e4 &    2.7611e4 &    1.8429e3 &    7.1496e2  &   2.5777 &    61.21\% \\[0.1cm] \hline
 \end{tabular}
\caption{Minimal and maximal eigenvalues, condition numbers, stiffness reduction ratios, and percentages for the 2D Laplace eigenvalue problem on the L-shaped domain when using Galerkin FEM and softFEM with $p\in\{1,2,3\}$ and an unstructured mesh.}
\label{tab:femlshape} 
\end{table}

Tables \ref{tab:fem2dum} and \ref{tab:femlshape} report the smallest and largest eigenvalues, the condition numbers, the stiffness reduction ratios and the percentages for the 2D Laplace eigenvalue problem on the unit square domain and the L-shaped domain{, respectively. We use} Galerkin FEM and softFEM with $p\in\{1,2,3\}$ and an unstructured mesh. Once again we observe that softFEM is capable to reduce significantly the stiffness of the resulting matrix on unstructured meshes as well.

\section{Proof of Theorem~\ref{thm:coe}} \label{sec:ana}

In this section{,} we prove Theorem~\ref{thm:coe} which establishes the coercivity
of the bilinear form $\hat a(\cdot,\cdot)$ under the condition that the softness
parameter $\eta\in [0,\eta_{\max})$ for some real number $\eta_{\max}$ depending on the
polynomial degree $p$ and the type of mesh.
To this purpose{,} we first establish some useful discrete trace
inequalities.

\subsection{Discrete trace inequalities} 

For a natural number $m\in\polN$, we define the sets $I_m:=\{0,\ldots,m\}$,
$\partial I_m:=\{0,m\}$, and 
$I_m^0:=I_m\setminus\partial I_m$. Let $p\in\polN$ be the polynomial degree. We are going to
consider the Gauss--Lobatto rule with $(p+2)$ points 
(see, for example, \cite{quarteroni2007numerical}), which is exact for polynomials
of degree at most $(2p+1)$. The weights are denoted $\{\varpi_j\}_{j\in I_{p+1}}$
and the nodes in $[-1,1]$ are denoted $\{\xi_j\}_{j\in I_{p+1}}$. Recall that 
\begin{equation} \label{eq:trace_1d}
\varpi_0 = \varpi_{p+1}= \frac{2}{(p+1) (p+2)}, \qquad \varpi_j = \frac{2}{(p+1) (p+2) (L_{p+1}(\xi_j))^2}, \quad \forall j \in I_{p+1}^0,
\end{equation}
where $L_{p+1}$ is the Legendre polynomial of degree $(p+1)$. 
For a univariate function $v$ that is $k$-times differentiable, 
we denote its $k$-th derivative as $v^{(k)}$. 

\begin{lemma}[Discrete trace inequality, 1D] \label{lem:trace1d}
Let  $\tau := [a, b]$ with $b > a$ and $\partial \tau = \{a, b\}$. 
Set $h_\tau:=b-a$.
For all $p\in\polN$ and all $k \in I_p$, the following holds:
\begin{equation} \label{eq:trace1d}
\| v^{(k)}  \|_{\partial \tau} \le C_1(k,p) h_\tau^{-1/2}  \ \| v^{(k)}  \|_{\tau}, \quad \forall v \in \mathbb{P}_p(\tau), 
\end{equation}
where $C_1(k,p) := \sqrt{ (p-k+1) (p-k+2) }$. Moreover, the constant $C_1(k,p) $ is sharp. In particular, for $p\ge1$ and $k=1$, we have
\begin{equation} \label{eq:trace1d1}
\| v'  \|_{\partial \tau} \le \sqrt{p(p+1)} h_\tau^{-1/2}  \ \| v'  \|_{\tau}, \quad \forall v \in \mathbb{P}_p(\tau). 
\end{equation}
\end{lemma}
\begin{proof}
It is clear that it suffices to prove \eqref{eq:trace1d} for $k=0$. Let $v \in \mathbb{P}_p(\tau)$.
We observe that $\|v\|_{\partial \tau}^2=v(a)^2+v(b)^2$.
Moreover, since $v^2$ is a polynomial of degree at most $2p$, 
it is integrated exactly by the Gauss--Lobatto quadrature with $(p+2)$ points. 
Considering the linear mapping from $\xi \in[-1,1]$ to $[a,b]$ with $x(\xi) := \frac{b-a}{2} \xi + \frac{a+b}{2}$ and setting $g_j:=x(\xi_j)$ for all $j\in I_{p+1}$, we have
\begin{equation*} 
\begin{aligned}
\| v  \|_{ \tau}^2 &= \int_a^b v^2(x) \ d x = \int_{-1}^1 v^2(x(\xi)) \frac{d x}{d \xi} \ d \xi \\
& = \frac{b-a}{2} \bigg( \frac{2 v^2(a) + 2 v^2(b)}{(p+1) (p+2)}  + \sum_{j\in I_{p+1}^0} 
\varpi_j v^2(g_j) \bigg) \\
& = \frac{b-a}{(p+1) (p+2)}\|v\|_{\partial \tau}^2 + \frac{b-a}{2} \sum_{j\in I_{p+1}^0} 
\varpi_j v^2(g_j) \ge
C_1(0,p)^{-2} h_\tau \|v\|_{\partial \tau}^2,
\end{aligned}
\end{equation*}
where we used that $g_0=a$ and $g_{p+1}=b$, 
the definition of $C_1(0,p)$ and $h_\tau$, and the fact that the
weights $\varpi_j$ are non-negative for all $j\in I_{p+1}^0$. This
proves~\eqref{eq:trace_1d} for $k=0$. Finally, that the inequality is sharp follows from
the fact that it is possible to find a nonzero polynomial in $\polP_p(\tau)$ that
vanishes at all the points $g_j$ for all $j\in I_{p+1}^0$.
\end{proof}

Let us now turn to the multi-dimensional case. We consider first the tensor-product case.
For simplicity{,} we focus on bounding the normal derivative on the boundary of a cuboid cell.
For a different result bounding any partial derivative on the boundary, we refer 
the reader to Remark~\ref{rem:trace}. 

\begin{lemma}[Discrete trace inequality, cuboid] \label{lem:tracegrad}
Let $\tau := [a_1, b_1] \times \ldots \times [a_d, b_d] \subset \mathbb{R}^d$, with $b_j > a_j$ for all $j\in \{1,\ldots,d\}$, be a cuboid with boundary $\partial \tau$ and outward normal $\bfs{n}_\tau$. Recall that $h_\tau^0:=\min_{i\in\{1,\ldots,d\}} (b_i-a_i)$ is the length of the smallest edge of $\tau$. Let $p\ge1$. The following holds:
\begin{equation} \label{eq:trace_cuboid}
\| \nabla v \cdot \bfs{n}_\tau  \|_{\partial \tau} \le \sqrt{p (p+1)} (h_\tau^0)^{-1/2} \| \nabla v  \|_{\tau}, \quad \forall v \in \mathbb{Q}_p(\tau){.}
\end{equation}
Moreover, the constant is sharp. Notice that \eqref{eq:trace_cuboid} coincides with \eqref{eq:trace1d1} for $d=1$. 
\end{lemma}

\begin{proof}
We present the proof in the 2D case ($d=2$); the general case is treated similarly. 
Let $v \in \mathbb{Q}_p(\tau)$. One can write $v(x,y) = \sum_{j_x, j_y \in I_p} \alpha_{j_x j_y} \psi_{j_x}(x) \psi_{j_y}(y)$, where $\{\psi_j\}_{j\in I_p}$ are basis functions of the univariate polynomial space of degree at most $p$. Moreover, we have $\partial \tau = \mathcal{F}_x \cup \mathcal{F}_y$. $\mathcal{F}_x$ contains two faces (located at $x=a_1, b_1$) and so does $\mathcal{F}_y$ (located at $y=a_2, b_2$). 
We consider the linear mappings $x:[-1,1]\to [a_1,b_1]$ and $y:[-1,1]\to [a_2,b_2]$. 
Let us first consider the two faces in $\mathcal{F}_x$. Since we are integrating the partial derivative of $v$ with respect to $x$, we consider a Gauss--Lobatto quadrature in $\tau$ obtained as the tensor-product of a Gauss--Lobatto quadrature with $(p+1)$ points in the $x$ variable and a Gauss--Lobatto quadrature with $(p+2)$ points in the $y$ variable. We use a superscript for the weights and nodes to indicate the number of points in the quadrature, and we set $g_j^x:=x(\xi_j^{p+1})$ for all $j\in I_{p}$ and $g_j^y:=y(\xi_j^{p+2})$ for all $j\in I_{p+1}$. Using the same arguments as in the proof of Lemma \ref{lem:trace1d}, we obtain
\begin{equation*} 
\begin{aligned}
\| \nabla v \cdot \bfs{n}_\tau  \|_{\mathcal{F}_x}^2 & = \int_{a_2}^{b_2} (\partial_x v|_{x=a_1})^2 \ d y +  \int_{a_2}^{b_2} (\partial_x v|_{x=b_1})^2 \ d y  \\
& = \frac{b_2 - a_2}{2} \sum_{l_x \in \partial I_{p}} \sum_{l_y \in I_{p+1}} \varpi_{l_y}^{p+2} \bigg( \sum_{j_x, j_y \in I_p} \alpha_{j_x j_y} \psi_{j_x}' (g_{l_x}^x) \psi_{j_y} (g_{l_y}^y)  \bigg)^2 \\
& = \frac{b_2 - a_2}{2}  \frac{p(p+1)}{2} \sum_{l_x \in \partial I_{p}} \sum_{l_y \in I_{p+1}}  \varpi_{l_x}^{p+1} \varpi_{l_y}^{p+2} \bigg( \sum_{j_x, j_y \in I_p} \alpha_{j_x j_y} \psi_{j_x}' (g_{l_x}^x) \psi_{j_y} (g_{l_y}^y)  \bigg)^2 \\
& \le \frac{b_2 - a_2}{2}  \frac{p(p+1)}{2} \sum_{l_x\in I_p, l_y \in I_{p+1}}  \varpi_{l_x}^{p+1} \varpi_{l_y}^{p+2} \bigg( \sum_{j_x, j_y \in I_p} \alpha_{j_x j_y} \psi_{j_x}' (g_{l_x}^x) \psi_{j_y} (g_{l_y}^y)  \bigg)^2 \\
& = \frac{p(p+1)}{b_1-a_1}  \| \partial_x v  \|_{ \tau}^2.
\end{aligned}
\end{equation*}
Similarly{,} we have 
\begin{equation*} 
\| \nabla v \cdot \bfs{n}_\tau  \|_{\mathcal{F}_y}^2 \le \frac{p(p+1)}{b_2-a_2}  \| \partial_y v  \|_{ \tau}^2.
\end{equation*}
Summing the above two inequalities and recalling the definition of $h_\tau^0$ gives
\begin{equation} 
\| \nabla v \cdot \bfs{n}_\tau  \|_{\partial \tau}^2 \le  \frac{p(p+1)}{b_1-a_1}  \| \partial_x v  \|_{ \tau}^2 + \frac{p(p+1)}{b_2-a_2}  \| \partial_y v  \|_{ \tau}^2 \le \frac{p (p+1) }{h_\tau^0} \| \nabla v  \|_{ \tau}^2.
\end{equation}
Taking square roots completes the proof for $d=2$. Finally, the constant is sharp since the upper bound in~\eqref{eq:trace_cuboid} can be attained by univariate functions. 
\end{proof}

Finally{,} we consider the case of a simplex.

\begin{lemma}[Discrete trace inequality, simplex] \label{lem:gradsp}
Let $\tau$ be a simplex in $\mathbb{R}^d$, $d\ge2$, with boundary $\partial\tau$ and outward normal $\bfs{n}_\tau$. Recall that $h_\tau^0:= \frac{d|\tau|}{|\partial\tau|}$. Let 
$p\ge1$. The following holds:
\begin{equation} \label{eq:trace_simplex}
\| \nabla v \cdot \bfs{n}_\tau  \|_{\partial \tau} \le \sqrt{p(p+d-1)} (h_\tau^0)^{-1/2}  \| \nabla v  \|_{\tau}, \quad \forall v \in \mathbb{P}_p(\tau).
\end{equation}
\end{lemma}

\begin{proof}
Let $v \in \mathbb{P}_p(\tau)$, let $F$ be a face of $\tau$ and set 
$\bfs{n}_F:=\bfs{n}_{\tau|F}$. Then $w := \nabla v \cdot \bfs{n}_F \in \mathbb{P}_{p-1}(\tau)$.
Applying the discrete trace inequality from \cite{warburton2003constants} yields
\begin{equation*}
\| w \|_{L^2(F)}^2 \le p(p+d-1) \frac{|F|}{d|\tau|}  \| w \|_{L^2(\tau)}^2.
\end{equation*}
Since $|\nabla v \cdot \bfs{n}_F| \le \|\nabla v\|_{\ell^2}$ (the Euclidean norm of $\nabla v$), we infer that
\begin{equation*}
\| \nabla v\cdot \bfs{n}_F \|_{L^2(F)}^2 \le p(p+d-1) \frac{|F|}{d|\tau|}  \| \nabla v \|_{L^2(\tau)}^2.
\end{equation*}
Summing over the faces of $\tau$, taking square roots, and recalling the definition of $h_\tau^0$ conclude the proof.
\end{proof}

\begin{remark}[Lemma~\ref{lem:trace1d}] \label{rem:trace}
Using the Gauss--Lobatto nodes and their tensor-products to prove discrete trace inequalities is a known technique. The result of Lemma~\ref{lem:trace1d} however slightly differs from previous results from the literature and provides a sharper constant. 
For instance, for $p\ge1$, $d=1$ and $k=0$, \cite[Thm.~2]{warburton2003constants}  leads to the constant $\sqrt{2(p+1)^2}$ and \cite[Lemma 3.1]{burman2007continuous} to the constant $\sqrt{(p+1)(2p+1)}$, which are both less sharp than $C_1(0,p)$ in \eqref{eq:trace1d}. Notice also that \eqref{eq:trace_simplex} with $d=1$ leads to 
$\| v'  \|_{\partial \tau} \le \sqrt{2p^2} h_\tau^{-1/2} \| v'  \|_{\tau}$ which is again less sharp that \eqref{eq:trace1d1} for $p\ge2$.
Finally, we have the following multidimensional extension of Lemma~\ref{lem:trace1d} in a cuboid; the proof is omitted for brevity and follows arguments similar to those above. 
Let $\tau := [a_1, b_1] \times \ldots \times [a_d, b_d] \subset \mathbb{R}^d$, with $b_j > a_j$ for all $j\in \{1,\ldots,d\}$, be a cuboid with boundary $\partial \tau$. Let $p\ge1$. For any multi-index $(\bfs{k}) = (k_1, \ldots, k_d)$ with $k_j \in I_p$ for all $j\in\{1,\ldots,d\}$, denoting the $\bfs{k}$-th partial derivative of $v$ as $v^{(\bfs{k})}$, the following holds:
\begin{equation}
\| v^{(\bfs{k})}  \|_{\partial \tau} \le C_d(\bfs{k},p,\tau)  \ \| v^{(\bfs{k})}  \|_{\tau}, \quad \forall v \in \mathbb{Q}_p(\tau),
\end{equation}
with $C_d(\bfs{k},p,\tau) := \sqrt{  \sum_{j\in \{1,\ldots ,d\}} \frac{  (p-k_j+1) (p-k_j+2) }{ b_j - a_j } }.$
Moreover, the constant $C_d(\bfs{k},p,\tau)$ is sharp.
\end{remark}

\subsection{Coercivity proof} 

We can now give the proof of Theorem \ref{thm:coe}.

\begin{proof}[\textbf{Proof of Theorem \ref{thm:coe}}]
\textup{(i)} Tensor-product meshes. 
%
For all $F\in\mathcal{F}_h^i$, let $\mathcal{T}_F$ be the set collecting the two mesh elements sharing $F$. For all $w^h\in V_p^{h}$, we have
\[
s(w^h,w^h) = \sum_{F \in \mathcal{F}_h^i} \kappa_F h_F \|\lsem \nabla w^h \cdot \bfs{n} \rsem\|_F^2
\le 2\sum_{F \in \mathcal{F}_h^i}\sum_{\tau\in \mathcal{T}_F} \kappa_F h_F \|\nabla w^h|_\tau \cdot \bfs{n}_\tau\|_F^2.
\]
Since $h_F=\min_{\tau\in \mathcal{T}_F} h_\tau^0$ and 
$\kappa_F=\min_{\tau\in \mathcal{T}_F} \kappa_\tau$ (see \eqref{eq:def_F_based}) and exchanging the order of the two summations, we infer that
\[ 
s(w^h,w^h) \le 2 \sum_{\tau\in \mathcal{T}_h} \kappa_\tau h_\tau^0
\|\nabla w^h \cdot \bfs{n}_\tau\|_{\partial\tau}^2.
\]
Applying Lemma \ref{lem:tracegrad} yields
\begin{equation*}
s(w^h,w^h) \le 2 p(p+1)\sum_{\tau\in \mathcal{T}_h} \kappa_\tau \|\nabla w^h\|_\tau^2
\le 2p(p+1)a(w^h,w^h).
\end{equation*}
Recalling that $\hat a(\cdot,\cdot) = a(\cdot,\cdot) - \eta s(\cdot,\cdot)$ with $\eta>0$, we conclude that 
\begin{equation} \label{eq:hat_a_a}
\hat a(w^h, w^h)  \ge (1 - 2 p(p+1) \eta) a(w^h,w^h).
\end{equation}
%
\\
\textup{(ii)} Simplicial meshes. The proof is similar but we now invoke Lemma~\ref{lem:gradsp} instead of Lemma \ref{lem:tracegrad}. 
\end{proof}

\section{Concluding remarks} \label{sec:conclusion} 

In this work, we have shown by mathematical analysis and numerical experiments the benefits of tempering the stiffness of the Galerkin FEM approximation of {second-order} elliptic spectral problems. {T}he idea is to subtract a least-squares penalty on the gradient jumps across the mesh interfaces from the stiffness bilinear form. This novel approximation technique has been named softFEM since it reduces the stiffness of the problem. SoftFEM is formulated in terms of one softness parameter for which we provided an admissible range of values to maintain coercivity on both tensor-product and simplicial meshes{. W}e also gave a practical choice of the softness parameter that leads to superconvergence for linear softFEM in 1D and to attractive numerical performances in more general situations. The main feature of softFEM is that it preserves the optimal accuracy of the eigenvalues in the low-frequency region, while at the same time improving significantly the accuracy in the high-frequency region. The main explanation for this improvement is, as illustrated numerically in our experiments, that in the high-frequency region the standard Galerkin FEM approximation tends to store a substantial amount of energy for the eigenfunctions in the form of gradient jumps across the mesh interfaces. Another very important advantage of softFEM that we illustrated in several settings is its ability to offer a sizable reduction of the conditioning of the stiffness matrix. The optimal value of the asymptotic stiffness reduction ratio increases linearly with the polynomial degree  and fairly close values to those predicted theoretically are recovered in our various numerical experiments.

As for future work, a first possible direction is the generalization to other {differential operators, such as the biharmonic operator. Two possible approaches for the FEM spectral approximation are the mixed (see, e.g., \cite{andreev2005postprocessing,deng2019optimal}) and the primal (see, e.g., \cite{bms2015}) formulations.} 
\begin{figure}[h!]
\hspace{-0.5cm}
\includegraphics[height=5.5cm]{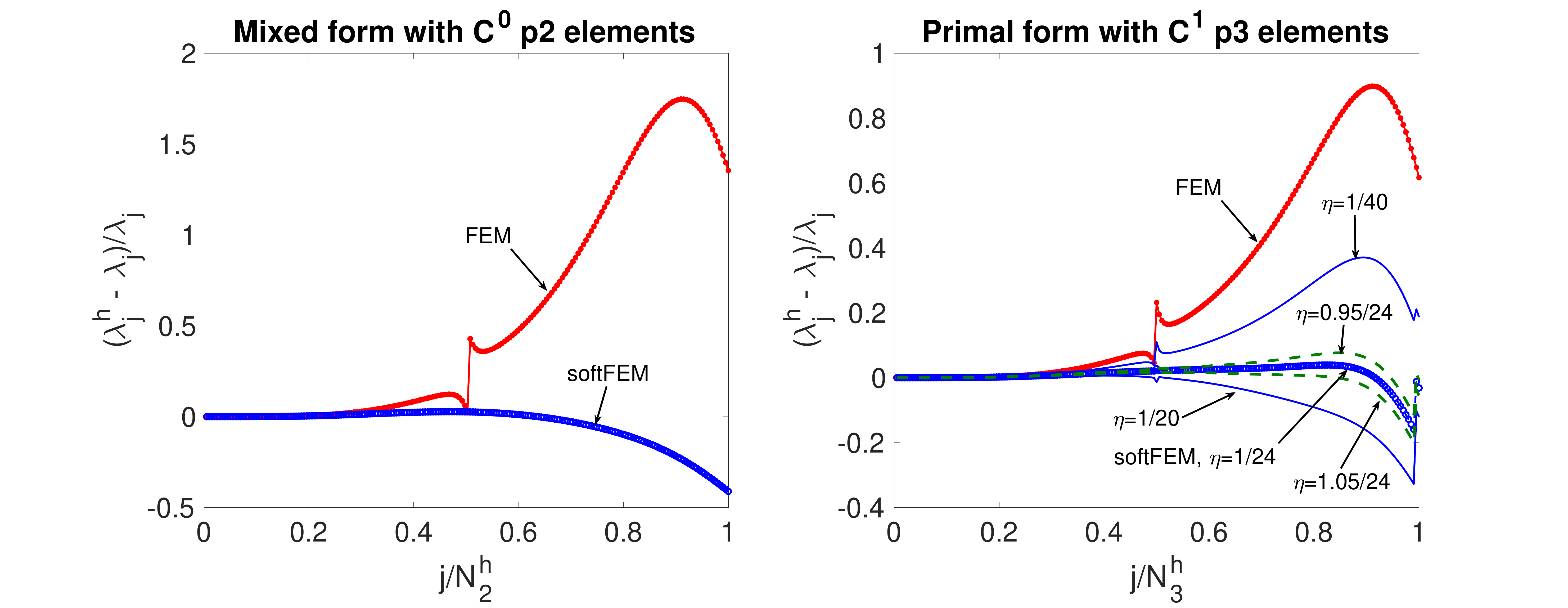} 
\vspace{-0.6cm}
\caption{FEM and softFEM spectral approximations of the 1D biharmonic eigenvalue problem using mixed (left) and primal (right) formulations on a uniform mesh composed of $100$ elements.}
\label{fig:biha}
\end{figure}
{Figure \ref{fig:biha} shows the FEM and softFEM spectral approximations of the 1D biharmonic eigenvalue problem: Find an eigenpair $(\lambda, u)$ such that
$
\Delta^2 u = u^{(4)} = \lambda u
$
in $\Omega = (0,1)$ with the simply supported plate boundary conditions $u = u'' = 0$ on $\partial \Omega$. 
For the mixed formulation, we decompose $u^{(4)} = \lambda u$ as $v'' = \lambda u$ and $u'' - v = 0$. We then apply FEM and softFEM, as developed in Section 2, to the decomposed problem with $C^0$ quadratic elements (notice that softFEM is employed for both equations). 
The left plot in Figure \ref{fig:biha} shows that softFEM maintains the same advantageous features of softFEM as for the second-order operator. In particular, softFEM reduces significantly the high-frequency spectral errors.
For the primal formulation, we consider the bilinear form $\hat a(v, w) = (v'', w'')_\Omega - \eta s(v, w)$ with the softness bilinear form $s(v,w)  := \sum_{F \in \mathcal{F}_h^i}  h_F (\lsem v'' \rsem, \lsem w'' \rsem )_F$.
The right plot of Figure \ref{fig:biha} shows the comparison of FEM and softFEM with various softness parameters when using $C^1$ cubic splines. The optimal choice for the softness parameter is $\eta = \frac{1}{24}$, that is, $\eta=\frac{1}{2p(p+1)}$ with $p=3$. Notice that this optimal value is different from the one found for the second-order elliptic operator ($\frac{1}{2(p+1)(p+2)}$). 
Further analysis of the optimality parameter along with the error analysis is postponed to future work.
}

Another future work direction is the generalization to other discretization methods, such as isogeometric analysis (leading to softIGA) and discontinuous Galerkin methods.
{Preliminary numerical tests indicate that softIGA has the same features as softFEM: it reduces the stiffness and condition numbers and it improves high-frequency spectral accuracy. More details will be reported in future work.}
Finally, the stiffness reduction by softFEM lends itself naturally to tempering the CFL condition in explicit time-marching schemes applied to time-dependent PDEs. 



\bibliographystyle{siamplain}

\end{document}